\documentclass[1p,final]{elsarticle}
\usepackage{amsfonts,color,morefloats,pslatex}
\usepackage{amssymb,amsthm, amsmath,latexsym}

\newtheorem{theorem}{Theorem}
\newtheorem{lemma}[theorem]{Lemma}

\newtheorem{corollary}[theorem]{Corollary}

\newtheorem{example}[theorem]{Example}

\newcommand\myatop[2]{\genfrac{}{}{0pt}{}{#1}{#2}}

\newcommand{\rank}{{\mathrm{rank}}}
\newcommand{\lcm}{{\mathrm{lcm}}}

\newcommand{\gf}{{\mathrm{GF}}}

\newcommand{\support}{{\mathrm{Suppt}}}

\newcommand{\dist}{{\mathrm{dist}}}

\newcommand{\wt}{{\mathtt{wt}}}

\newcommand{\cP}{{\mathcal{P}}} 
\newcommand{\cB}{{\mathcal{B}}}

\newcommand{\C}{{\mathcal{C}}}
\newcommand{\A}{{\mathfrak{A}}}

\newcommand{\bzero}{{\mathbf{0}}}
\newcommand{\bone}{{\mathbf{1}}}

\newcommand{\bD}{{\mathbb{D}}}

\begin{document}

\begin{frontmatter}



\title{A spectral characterisation of $t$-designs and its applications}


\author[cho]{Eun-Kyung Cho}
\ead{ekcho@pusan.ac.kr} 
\author[cding]{Cunsheng Ding}
\ead{cding@ust.hk}
\author[hyun]{Jong Yoon Hyun}
\ead{hyun33@kias.re.kr}


\address[cho]{Department of Mathematics, Pusan National University, Republic of Korea}                                                  

\address[cding]{Department of Computer Science
                                                  and Engineering, The Hong Kong University of Science and Technology,
                                                  Clear Water Bay, Kowloon, Hong Kong, China}                                                  
\address[hyun]{Korea Institute for Advanced Study (KIAS), Seoul, Republic of Korea}


\begin{abstract}
There are two standard approaches to the construction of $t$-designs. The first one is based on permutation 
group actions on certain base blocks. The second one is based on coding theory. The objective of this paper is to give a spectral characterisation of all $t$-designs by introducing a characteristic Boolean 
function of a $t$-design. The spectra of the characteristic functions of $(n-2)/2$-$(n, n/2, 1)$ Steiner systems are determined and properties of such designs are proved. Delsarte's characterisations of orthogonal arrays and $t$-designs, which are two special cases of 
Delsarte's characterisation of $T$-designs in association schemes, are slightly 
extended into two spectral characterisations. Another characterisation of $t$-designs 
by Delsarte and Seidel is also extended into a spectral one. These spectral characterisations are then compared with the new spectral characterisation 
of this paper.      
\end{abstract}


\end{frontmatter}

\section{Introduction}

Let $\cP$ be a set of $n \ge 1$ elements, and let $\cB$ be a set of $k$-subsets of $\cP$, where $k$ is
a positive integer with $1 \leq k \leq n$. Let $t$ be a positive integer with $t \leq k$. The pair
$\bD = (\cP, \cB)$ is called a $t$-$(n, k, \lambda)$ {\em design\index{design}}, or simply {\em $t$-design\index{$t$-design}}, if every $t$-subset of $\cP$ is contained in exactly $\lambda$ elements of
$\cB$. The elements of $\cP$ are called points, and those of $\cB$ are referred to as blocks.
We usually use $b$ to denote the number of blocks in $\cB$.  A $t$-design is called {\em simple\index{simple}} if $\cB$ does not contain repeated blocks. In this paper, we consider only simple 
$t$-designs.  A $t$-design is called {\em symmetric\index{symmetric design}} if $n = b$. It is clear that $t$-designs with $k = t$ or $k = n$ always exist. Such $t$-designs are {\em trivial}. In this paper, we consider only $t$-designs with $n > k > t$.
A $t$-$(n, k, \lambda)$ design is referred to as a {\em Steiner system\index{Steiner system}} if $t \geq 2$ and $\lambda=1$, and is denoted by $S(t, k, n)$.  

The existence and constructions of $t$-designs have been a fascinating topic of research for about one hundred and fifty years   
\cite{AK92,BJL,HBComb,HP03,MS77,Tonchev,Tonchevhb}. One fundamental construction is the group action approach 
\cite[Chapter III]{BJL}, which employs transitive or homogeneous permutation groups. The fatal limitation of 
this approach lies in the fact that highly transitive or homogeneous permutation groups 
other than the symmetric and alternating groups do not exist   
\cite[Chapter V]{BJL}. Another fundamental construction is based on error-correcting codes \cite{AK92,Tonchev,Tonchevhb}. This approach makes use of the automorphism group of a code or the Assmus-Mattson Theorem, and has also limitations. By now no infinite family of $4$-designs is directly constructed from codes. There are numerous constructions of $t$-designs with flexible parameters in the literature and important progresses on the existence of $t$-designs have been made \cite{Teirlinck,Wilson72a,Wilson72b,Wilson75}. 
A characterisation of $t$-designs was given in Delsarte's thesis and is a special case (the Johnson scheme case) of a characterisation of $T$-designs in association schemes 
\cite{Dels73}, which is not a spectral characterisation. 

The main objective of this paper is to present a spectral characterisation of $t$-$(n, k, \lambda)$ designs. This is done by studying the characteristic Boolean function of a $t$-$(n, k, \lambda)$ design. As one application of this characterisation, we will determine the spectra of the 
characteristic functions of $(n-2)/2$-$(n, n/2, 1)$ Steiner systems, and prove properties of such designs. We will 
also show two applications of $(n-2)/2$-$(n, n/2, 1)$ Steiner systems in coding theory. 
As a byproduct, we will extend a characterisation of $t$-designs by Delsarte and another 
one by Delsarte and Seidel into spectral characterisations and will then compare them with the spectral characterisation of this paper. It will be shown that the characterisation of 
$t$-designs presented in this paper is much simpler. 

\section{Krawtchouk polynomials and their properties} 

In this section, we introduce Krawchouk polynomials and summarize their properties, which will be 
needed in subsequent sections. A proof of these results could be found in \cite[Ch. 5, Sections 2 and 7]{MS77}. 

Let $n$ be a positive integer, and let $x$ be a variable taking nonnegative values. The Krawtchouk polynomial is defined by 
\begin{eqnarray}
P_k(x)=\sum_{j=0}^k (-1)^j \binom{x}{j} \binom{n-x}{k-j} 
\end{eqnarray} 
where $0 \leq k \leq n$ and 
$$ 
\binom{x}{i}=\frac{x(x-1) \cdots (x-i+1)}{i!}.  
$$ 
It is easily seen that 
\begin{eqnarray}
(1+z)^{n-x} (1-z)^x = \sum_{k=0}^n P_k(x)z^k.  
\end{eqnarray} 

The following alternative expressions will be useful later. 

\begin{theorem}\label{thm-KP191}
Let notation be the same as before. 
\begin{itemize}
\item $P_k(x)=\sum_{j=0}^k (-2)^j \binom{n-j}{k-j}  \binom{x}{j}. $
\item  $P_k(x)=\sum_{j=0}^k (-1)^j 2^{k-j} \binom{n-k+j}{j}  \binom{n-x}{k-j}. $
\end{itemize} 
\end{theorem}

The orthogonality of Krawtchouk polynomials is documented below. 

\begin{theorem}
For nonnegative integers $r$ and $s$, 
\begin{eqnarray}
\sum_{i=0}^n \binom{n}{i}P_r(i) P_s(i) = 2^n \binom{n}{r} \delta_{r,s},    
\end{eqnarray} 
where $\delta_{r,s}=1$ if $r=s$ and  $\delta_{r,s}=0$ if $r \neq s$.  
\end{theorem}

\begin{theorem}
For nonnegative integers $r$ and $s$, 
$$ 
\binom{n}{i}  P_s(i) =  
\binom{n}{s}  P_i(s). 
$$
\end{theorem}

\begin{theorem}\label{thm-july164}
For nonnegative integers $r$ and $s$, 
\begin{eqnarray}
\sum_{i=0}^n P_r(i) P_i(s) = 2^n \delta_{r,s}.     
\end{eqnarray} 
\end{theorem}

\begin{theorem}\label{thm-july163}
Let $u \in \gf(2)^n$ with Hamming weight $\wt(u)=i$. Then 
$$ 
\sum_{\myatop{v \in \gf(2)^n}{\wt(v)=k}} (-1)^{u \cdot v} = P_k(i),  
$$ 
where $u \cdot v$ is the standard inner product of $u$ and $v$. 
\end{theorem}

The next theorem documents further basic properties of the Krawtchouk polynomials. 

\begin{theorem}\label{thm-ms21}
Let notation be the same as before. 
\begin{itemize}
\item $\sum_{k=0}^n \binom{n-k}{n-j} P_k(x)=2^j \binom{n-x}{j}.$ 
\item $P_k(i)=(-1)^iP_{n-k}(i)$, $0 \leq i \leq n$. 
\item $P_{n/2}(n/2)= (-1)^{n/4} \binom{n/2}{n/4}$ if $n \equiv 0 \pmod{4}$. 
\item $P_n(n)=(-1)^n$. 
\item $P_k(1)=\frac{n-2k}{n} \binom{n}{k}.$
\item $P_k(0)=\binom{n}{k}.$
\end{itemize}
\end{theorem} 

\begin{theorem}\label{thm-july161}
Let notation be the same as before. We have 
$$ 
P_k(x)=(-1)^k P_k(n-x). 
$$
\end{theorem} 

\begin{proof}
By definition, 
\begin{eqnarray*}
P_k(n-x) = \sum_{j=0}^k (-1)^j \binom{n-x}{j} \binom{x}{k-j}. 
\end{eqnarray*} 
Substituting $k-j$ with $i$, we get 
\begin{eqnarray*}
P_k(n-x) = \sum_{i=0}^k (-1)^{k-i} \binom{n-x}{k-i} \binom{x}{i} =(-1)^k P_k(x).   
\end{eqnarray*}
\end{proof}

\section{Basics of $t$-designs} 

In this paper, we will consider $t$-designs with the point set $\cP=\{1,2, \ldots, n\}$, where $n$ 
is a positive integer. For simplicity, we use $[i..j]$ to denote the set $\{i, i+1, \ldots, j\}$ 
for any two positive integers $i$ and $j$ with $i \leq j$. For an integer $i$ with $0 \leq i \leq 
n$, denote by $\binom{\cP}{i}$ the set of all $i$-subsets of $\cP$. 

We will need the following lemmas later \cite[p. 15]{BJL}. 

\begin{lemma}\label{lem-designbasicp} 
Let $\bD$ be a $t$-$(n, k, \lambda)$ design. Let $s$ be an integer with $1 \leq s \leq t \leq k$. 
Then $\bD$ is also an $s$-$(n, k, \lambda_s)$ design, where 
\begin{eqnarray}\label{eqn-fundms}
\lambda_s = \lambda \frac{\binom{n-s}{t-s}}{\binom{k-s}{t-s}}. 
\end{eqnarray} 
In addition, 
\begin{eqnarray}
b := \lambda_0 = \lambda \frac{\binom{n}{t}}{\binom{k}{t}}  
\end{eqnarray} 
is the number of blocks in the design $\bD$.  
\end{lemma} 

Let $\bD=(\cP, \cB)$ be a $t$-$(n, k, \lambda)$ design. Let $\overline{\cB}$ be the set of 
the complements of all the blocks $B$ in $\cB$, and let $\overline{\bD}=(\cP, \overline{\cB})$. 

\begin{lemma}\label{lem-complementdesign} 
Let $\bD=(\cP, \cB)$ be a $t$-$(n, k, \lambda)$ design. Then $\overline{\bD}=(\cP, \overline{\cB})$ is 
an $s$-$(n, n-k, \overline{\lambda}_s)$ design for all $1 \leq s \leq t$, where 
\begin{eqnarray}
\overline{\lambda}_s = \sum_{i=0}^s (-1)^i \binom{s}{i} \lambda_i.  
\end{eqnarray}  
In particular, 
$$ 
\overline{\lambda}_t:=\frac{\lambda \binom{n-k}{t}}{\binom{k}{t}}. 
$$
\end{lemma} 

The design $\overline{\bD}$ is called the \textit{complementary design} of $\bD$. We will employ the two 
forgoing lemmas later.  

Let $\bD=(\cP, \cB)$ be a $t$-$(n, k, \lambda)$ design. Let $i$ and $j$ be two nonnegative integers, 
and let $X=\{p_1, p_2, \ldots, p_{i+j}\}$ be a set of distinct points. Denote by $\lambda_{(i,j)}$ 
the number of blocks $B_\ell$ of $\bD$ such that 
$$ 
B_\ell \cap \{p_1, p_2, \ldots, p_{i+j}\}=Y:=\{p_1, p_2, \ldots, p_i\}.  
$$ 
These numbers $\lambda_{(i,j)}$ are called \textit{block intersection numbers}, 
and depend on not only $i$ and $j$, but also the specific points in $X$. 
However, under certain conditions these intersection numbers are dependent 
of $i$ and $j$ only. Specifically, we have the following \cite[p. 101]{BJL}. 

\begin{theorem}\label{thmBJL101}
Let $\bD=(\cP, \cB)$ be a $t$-$(n, k, \lambda)$ design. Let $i$ and $j$ be nonnegative integers. 
Then the number $\lambda_{(i,j)}$ depends only on $i$ and $j$, but not the points in $X$ and $Y$ if $i+j 
\leq t$ or $\lambda=1$ and $X$ is contained in some block of $\bD$.   
\end{theorem}

We first have the following result.   

\begin{lemma}\label{lem-hdbooks} 
Let $\bD=(\cP, \cB)$ be a $t$-$(n, k, \lambda)$ design. Let $i$ and $j$ be nonnegative integers. 
If $0 \leq i + j \leq t$, then 
$$ 
\lambda_{(i,j)}=\frac{\lambda \binom{n-i-j}{k-i}}{\binom{n-t}{k-t}}. 
$$ 
\end{lemma} 

The following facts about these $\lambda_{(i,j)}$ are well known: 
\begin{itemize}
\item $\lambda_{(i,0)}=\lambda_i$ for $0 \leq i \leq t$. 
\item $\lambda_{(0,i)}=\overline{\lambda}_i$ for $0 \leq i \leq t$.  
\item $\lambda_{(i, j)} = \lambda_{(i, j+1)} + \lambda_{(i+1, j)}$ for $i+j \leq t$, which is called the \textit{triangular 
formula}.   
\end{itemize} 

Consider now a $t$-$(n, t+1, 1)$ design $\bD$. Let $X$ be any 
block of $\bD$ and let $Y$ be an $i$-subset of $X$. Denote by $\lambda_{(i, t+1-i)}(X, Y)$ the number of blocks $B_j$ in $\cB$ such that 
$$ 
B_j \cap X = Y, 
$$ 
where $X$ is a block in $\cB$ and $Y$ is $i$-subset of $X$. 
By Theorem \ref{thmBJL101}, these numbers $\lambda_{(i, t+1-i)}(X, Y)$ depend only on $i$ and $t$. 
Hence, the triangular formula above still holds for $0 \leq i +j \leq t+1$ \cite[p. 9]{AK92}. 

We have then the following theorem. 

\begin{theorem}\label{thm-mymay22}
Let $\bD$ be a $t$-$(n, t+1, 1)$ design $\bD$. Let $X$ be any 
block of $\bD$ and let $Y$ be a $j$-subset of $X$. Then 
\begin{eqnarray}
\lambda_{(t-(j-1), j)}(X, Y)=\frac{ (-1)^{j-1} \sum_{\ell=0}^{j-1} (-1)^\ell \binom{n-t}{\ell+1}}{n-t} 
+(-1)^j 
\end{eqnarray} 
for $1 \leq j \leq t+1$. 
\end{theorem}  

\begin{proof}
With the triangular formula, we have 
$$ 
\lambda_{(t-(j-1), j)}(X, Y) = (-1)^j \lambda_{(t+1, 0)}(X, \emptyset) + \sum_{\ell=0}^{j-1} (-1)^\ell \lambda_{(t-(j-1)+\ell, j-1-\ell)}
$$ 
for $1 \leq j \leq t+1$. By definition, $\lambda_{(t+1, 0)}(X, \emptyset)=1$. The desired conclusion then follows 
from Lemma \ref{lem-hdbooks}. 
\end{proof}

\section{A spectral characterization of $t$-designs} 

A Boolean function with $n$ variables is a function $f(x_1, x_2, \ldots, x_n)$ from $\gf(2)^n$ 
to $\{0, 1\}$, which is viewed as a subset of the set of real numbers. In other words, Boolean 
functions in this paper are special real-valued functions unless otherwise stated.  
Let $x=(x_1, x_2, \ldots, x_n)$. The first kind of Walsh transform $\hat{f}$ of $f$ 
is defined by 
\begin{eqnarray}\label{eqn-firstWalshT}
\hat{f}(w)=\sum_{x \in \gf(2)^n} f(x) (-1)^{w\cdot x}, 
\end{eqnarray} 
where $w=(w_1, w_2, \ldots, w_n) \in \gf(2)^n$, $w \cdot x =\sum_{i=1}^n w_ix_i$ is the standard 
inner product in the vector space $\gf(2)^n$. The multiset $\{\hat{f}(w): w \in \gf(2)^n\}$ 
is called the spectra of $f(x)$. 
It is easily verified that the 
inverse transform is given by 
\begin{eqnarray}\label{eqn-101010}
f(x)=\frac{1}{2^n} \sum_{w \in \gf(2)^n} \hat{f}(w) (-1)^{w\cdot x}.  
\end{eqnarray}  

The support $\support(f)$ of $f$ is defined by 
$$ 
\support(f)=\{u \in \gf(2)^n: f(u)=1\} \subseteq \gf(2)^n. 
$$ 
The mapping $f \mapsto \support(f)$ is a one-to-one correspondence from the set of all Boolean functions 
with $n$ variables to the power set of $\gf(2)^n$. The weight $\wt(f)$ of $f$ is defined to be the cardinality of 
$\support(f)$.

The support of a vector $b=(b_1, b_2, \ldots, b_n) \in \gf(2)^n$ is defined by 
$$ 
\support(b)=\{1 \leq i \leq n: b_i=1\} \subseteq [1..n],  
$$ 
where $[i..j]$ denotes the set $\{i, i+1, \ldots, j\}$ for two nonnegative integers $i$ and $j$ with $i \leq j$. It is obvious that the mapping 
\begin{eqnarray}
\varphi: b \mapsto \support(b) 
\end{eqnarray} 
is a one-to-one correspondence from $\gf(2)^n$ to $2^{[1..n]}$, which denotes the power set of $[1..n]$. 

Let $\cP=[1..n]$ be a set of $n \ge 1$ elements, and let $\cB=\{B_i: 1 \leq i \leq b\}$ be 
a set of $k$-subsets of $\cP$, where $k$ is a positive integer with $1 \leq k \leq n$, and $b$ is a  
positive integer. The pair $\bD = (\cP, \cB)$ is called an incidence structure. The \emph{characteristic function} 
of the incidence structure $\bD$, denoted by $f_{\bD}(x)$, is the Boolean function of $n$ variables with support 
\begin{eqnarray}\label{eqn-supportIncidance}
\left\{\varphi^{-1}(B_i): 1 \leq i \leq b \right\}.
\end{eqnarray}

We are now ready to present a spectral characterization of $t$-designs. 

\begin{theorem}\label{thm-maincharacterisationthm}
Let $\bD=(\cP, \cB)$ be an incidence structure, where the point set $\cP=[1..n]$,  
the block set $\cB=\{B_1, B_2, \ldots, B_b\}$, the block size $|B_i|$ is $k$, and $k$ and $b$ 
are positive integers. Then $\bD$ is a $t$-$(n, k, \lambda)$ design if and only if 
for each integer $h$ with $0 \leq h \leq t$, 
\begin{eqnarray}\label{eqn-necsuff}
\hat{f}_{\bD}(w)= \frac{\lambda \sum_{i=0}^h (-1)^i \binom{h}{i} \binom{n-h}{k-i} }{\binom{n-t}{k-t}} 
= \frac{\lambda P_k(h) }{\binom{n-t}{k-t}} 
\end{eqnarray} 
for all $w \in \gf(2)^n$ with $\wt(w)=h$. 
\end{theorem} 

\begin{proof}
We first prove the necessity of the conditions in (\ref{eqn-necsuff}). Assume that $\bD$ is a $t$-$(n, k, \lambda)$ design. 
Let $w$ be 
a vector in $\gf(2)^n$ with $\wt(w)=h$, where $0 \leq h \leq t$. The inner product $w \cdot \varphi^{-1}(B_i)$ 
is given by 
$$ 
w \cdot \varphi^{-1}(B_i) = |\support(w) \cap B_i| \bmod 2.  
$$ 
Note that $|\support(w) \cap B_i|$ takes on only values in the following set 
$$ 
\{h, h-1, \ldots, 1, 0\}. 
$$
It then follows from Lemma \ref{lem-hdbooks} that 
$$ 
|\{1 \leq j \leq b: |\support(w) \cap B_j|=i\}|=\binom{h}{i} \lambda_{(i, h-i)}
=\frac{\lambda \binom{h}{i} \binom{n-h}{k-i}}{\binom{n-t}{k-t}},  
$$ 
where $0 \leq i \leq h$. Note that $h \leq t \leq k$. By convention, $\binom{h}{i}=0$ 
if $i >h$. 
We now deduce that 
\begin{eqnarray*}
\hat{f}_{\bD}(w) &=& 
\sum_{i=0}^h (-1)^i |\{1 \leq j \leq b: |\support(w) \cap B_j|=i\}| \\
&=& \frac{\lambda \sum_{i=0}^h (-1)^i \binom{h}{i} \binom{n-h}{k-i} }{\binom{n-t}{k-t}} \\ 
&=& \frac{\lambda \sum_{i=0}^k (-1)^i \binom{h}{i} \binom{n-h}{k-i} }{\binom{n-t}{k-t}} \\
&=& \frac{\lambda P_k(h) }{\binom{n-t}{k-t}} . 
\end{eqnarray*}
This proves the necessity of the conditions in (\ref{eqn-necsuff}).

We now prove the sufficiency of the conditions in (\ref{eqn-necsuff}) 
by induction. 
We first prove that $\bD$ is a $1$-$(n, k, \lambda_1)$ design. For each $w$ in $\gf(2)^n$ 
with weight $1$, the conditions in (\ref{eqn-necsuff}) in the case $h=1$ say that 
$$ 
\hat{f}_{\bD}(w)=\lambda \frac{P_k(1)}{\binom{n-t}{k-t}}.  
$$ 
The first alternative expression of the Krawtchouk polynomial given in Theorem \ref{thm-KP191} yields 
$$ 
P_k(1)=\binom{n}{k} -2 \binom{n-1}{k-1}. 
$$ 
We have then 
\begin{eqnarray}\label{eqn-May19a}
\hat{f}_{\bD}(w)=\lambda \frac{\binom{n}{k} - 2 \binom{n-1}{k-1}}{\binom{n-t}{k-t}}.  
\end{eqnarray} 
By the definition of binomial coefficients, 
\begin{eqnarray*}
\binom{n}{k}  \binom{k}{t} 
&=& \frac{n!}{k! (n-k)!}    \frac{k!}{t! (k-t)!} \\
&=&  \frac{n!}{t! (n-k)! (k-t)!} \\
&=&  \frac{n!}{t! (n-t)!}   \frac{(n-t)!}{(n-k)! (k-t)!}  \\
&=& \binom{n}{t}  \binom{n-t}{k-t}.    
\end{eqnarray*}  
Consequently, 
\begin{eqnarray}\label{eqn-May19b}
\frac{\binom{n}{k}}{\binom{n-t}{k-t}} = \frac{\binom{n}{t}}{\binom{k}{t}}. 
\end{eqnarray} 
Similarly, one can prove that 
\begin{eqnarray}\label{eqn-May19c}
\frac{\binom{n-1}{k-1}}{\binom{n-t}{k-t}} = \frac{\binom{n-1}{t-1}}{\binom{k-1}{t-1}}. 
\end{eqnarray} 
Plugging (\ref{eqn-May19b}) and  (\ref{eqn-May19c}) into (\ref{eqn-May19a}), we obtain 
$$ 
\hat{f}_{\bD}(w)=\lambda_0-2\lambda_1=b-2\lambda_1. 
$$
Suppose that $\support(w)=\{i\}$, where $1 \leq i \leq n$. Assume that $i$ is incident 
with $u$ blocks in $\cB$. It then follows from the definition of $\hat{f}_{\bD}(w)$ that 
$$ 
\hat{f}_{\bD}(w)=b-2u. 
$$
Consequently, $u=\lambda_1$, which is independent of $i$. By definition, $\bD$ is a 
$1$-$(n, k, \lambda_1)$ design. 

Suppose now that $\bD$ is an $s$-$(n, k, \lambda_s)$ design for all 
$s$ with $1 \leq s \leq h-1$ and $h \leq t$. We now prove that it is also an $h$-$(n, k, \lambda_h)$ 
design. Let $w$ be a vector in $\gf(2)^n$ with Hamming weight $h$. 
Let
$$ 
e = |\{1 \leq j \leq b: |\support(w) \cap B_j| =h\}|.
$$ 
Then by induction hypothesis,  
we have 
\begin{eqnarray*}
\lefteqn{|\{1 \leq j \leq b: |\support(w) \cap B_j| = i\}|} \\ 
&=& \binom{h}{i} \lambda_{i} - 
\binom{h}{i+1}\lambda_{i+1} + \cdots + (-1)^{h-1-i} \binom{h}{h-1} \lambda_{h-1} 
+(-1)^{h-i} \binom{h}{h} e\\
&=& \sum_{j=i}^{h-1} \binom{h}{j}(-1)^{j-i} \lambda_{j}+(-1)^{h-i}e
\end{eqnarray*}
for all $i$ with $0 \leq i \leq h-1$. 
As a result, we obtain 
\begin{eqnarray*}
\hat{f}_{\bD}(w) 
&=& \sum_{i=1}^b (-1)^{\varphi^{-1}(B_i) \cdot w} \\
&=& \sum_{i=0}^h  (-1)^i  |\{1 \leq j \leq b: |\support(w) \cap B_j|=i\}|  \\
&=& (-1)^h  |\{1 \leq j \leq b: |\support(w) \cap B_j|=h\}| \\ 
& & + \sum_{i=0}^{h-1}  (-1)^i  |\{1 \leq j \leq b: |\support(w) \cap B_j|=i\}|   \\ 
&=& (-1)^h  e  + \sum_{i=0}^{h-1}  (-1)^i  \left\{\sum_{j=i}^{h-1} \binom{h}{j}(-1)^{j-i} \lambda_{j}+(-1)^{h-i}e \right\}  \\ 
&=& (-1)^h  e  
+ \sum_{i=0}^{h-1} \sum_{j=i}^{h-1} \binom{h}{j}(-1)^{j} \lambda_{j} +
\sum_{i=0}^{h-1} (-1)^h e \\
&=& (-1)^h  e  
+ \sum_{i=0}^{h-1} \sum_{j=i}^{h-1} \binom{h}{j}(-1)^{j} \lambda_{j} + h (-1)^{h} e \\
&=& (h+1) (-1)^{h} e + \sum_{i=0}^{h-1} \sum_{j=i}^{h-1} \binom{h}{j}(-1)^{j} \lambda_{j}.
\end{eqnarray*} 
By the conditions in (\ref{eqn-necsuff}), $\hat{f}_{\bD}(w)$ is a constant for all $w$ with $\wt(w)=h$. 
Since every quantity in the above equation is fixed except for $e$,
this value $e$ is also a constant for all $w$ with $\wt(w)=h$. 
Consequently, $\bD$ is an $h$-$(n, k, e)$ design.
By induction, $\bD$ is a $t$-$(n, k, \tilde{\lambda})$ design for some $\tilde{\lambda}$.
Thus for $w \in \gf(2)^{n}$ with Hamming weight $t$,
we have 
$$\hat{f}_{\bD}(w) = \frac{\tilde{\lambda} P_k(t) }{\binom{n-t}{k-t}} 
= \frac{\lambda P_k(t) }{\binom{n-t}{k-t}}$$
by the conditions in (\ref{eqn-necsuff}), and using the fact that $\bD$ is a $t$-$(n, k, \tilde{\lambda})$ design.
Thus we have $\tilde{\lambda} = \lambda$, and 
we can conclude that $\bD$ is a $t$-$(n, k, \lambda)$ design. 
The proof is then completed.    
\end{proof}

\begin{example}[Fano plane in finite geometry]\label{exam-fanofunc} 
Let $\cP=\{1, 2, 3, 4, 5, 6, 7\}$ and 
$$ 
\cB=\{\{1, 2, 3\}, \{1, 4, 5\}, \{1, 6, 7\}, \{2, 4, 7\}, \{2, 5, 6\},  \{3, 4, 6\}, \{3, 5, 7\}\}. 
$$ 
Then $\bD=(\cP, \cB)$ is a $2$-$(7, 3, 1)$ design, i.e., Steiner triple system $S(2, 3, 7)$. 

The characteristic function $f_{\bD}$ of $\bD$ is given by 
\begin{eqnarray*}
&&
 x_1  x_2  x_3  x_4  x_5  x_6  x_7 +  \\
&&  x_1  x_2  x_3  x_4 +  x_1  x_2  x_3  x_5 +  x_1  x_2  x_3  x_6 +  x_1  x_2  x_3  x_7 +  x_1  x_2  x_4  x_5 +
     x_1  x_2  x_4  x_7 +  x_1  x_2  x_5  x_6 +  \\
&&     x_1  x_2  x_6  x_7 +  x_1  x_3  x_4  x_5 +  x_1  x_3  x_4  x_6 +  x_1  x_3  x_5  x_7 +  x_1  x_3  x_6  x_7 +
     x_1  x_4  x_5  x_6 +  x_1  x_4  x_5  x_7  +  \\ 
&&     x_1  x_4  x_6  x_7 +  x_1  x_5  x_6  x_7 +  x_2  x_3  x_4  x_6 +  x_2  x_3  x_4  x_7 +
     x_2  x_3  x_5  x_6 +  x_2  x_3  x_5  x_7 +  x_2  x_4  x_5  x_6 +  \\ 
&&     x_2  x_4  x_5  x_7 +  x_2  x_4  x_6  x_7 +  x_2  x_5  x_6  x_7 + 
     x_3  x_4  x_5  x_6 +  x_3  x_4  x_5  x_7 +  x_3  x_4  x_6  x_7  +  x_3  x_5  x_6  x_7   + \\ 
&&      
 +  x_1  x_2  x_3 +  x_1  x_4  x_5 +  x_1  x_6  x_7 +  x_2  x_4  x_7 +  x_2  x_5  x_6 +  x_3  x_4  x_6 + x_3  x_5  x_7, 
\end{eqnarray*}
where the additions and multiplications are over $\gf(2)$. 

\end{example} 

\begin{example} 
Let $\cP=[1..12]$, and let $\bD=(\cP, \cB)$ be the Steiner system $S(5, 6, 12)$. Then the characteristic 
function of $\bD$ is given by 
\begin{eqnarray*}
f_{\bD}(x)= \sum_{(i_1, i_2, i_3, i_4, i_5, i_6) \in \cB} x_{i_1}x_{i_2}x_{i_3} x_{i_4}x_{i_5} x_{i_6} 
+ \sum_{(i_1, i_2, i_3, i_4, i_5, i_6, i_7) \in \binom{\cP}{7}} x_{i_1}x_{i_2}x_{i_3} x_{i_4}x_{i_5} x_{i_6}x_{i_7}, 
\end{eqnarray*} 
where $\binom{\cP}{7}$ denotes the set of all $7$-subsets of $[1..12]$. Hence, the algebraic form of 
$f_{\bD}(x)$ has $924$ terms, but looks interesting in the sense that it is compact and simple. 
\end{example}

\section{The algebraic normal form of the characteristic function $f_{\bD}$ of $t$-designs $\bD$} 

Recall that Boolean functions are real-valued functions taking on only 
the two integers $0$ and $1$. 
Let $f(x)$ be a Boolean function from $\gf(2)^n$ to $\{0,1\}$. Suppose that the support of $f$ 
is $\{v_1, \ldots, v_b\}$, where $b$ is a positive integer. Let $v_i=(v_{i,1}, v_{i,2}, \ldots, v_{i,n}) \in \gf(2)^n$ 
for each $i$, and let $\bar{v}_{i,j}=v_{i,j}+1 \in \gf(2)$ for all $i$ and $j$. Let $B(d,n)$ denote the set 
of all $d$-subsets of $[1..n]$.

By definition, 
\begin{eqnarray}\label{eqn-anf}
f(x) &=& \sum_{i=1}^b (x_1 +\bar{v}_{i, 1})(x_2 + \bar{v}_{i, 2}) \cdots (x_n + \bar{v}_{i, n}) \nonumber \\ 
     &=& \sum_{d=0}^n \sum_{\{i_1, \ldots, i_d\} \in B(d, n)} \left( \sum_{i=1}^b \prod_{j \in [1..n] \setminus \{i_1, \ldots, i_d\}} \bar{v}_{i,j}\right)   \prod_{j=1}^d x_{i_j},   
\end{eqnarray}
where all the additions and multiplications are over $\gf(2)$, and an empty product is defined to be $1$ by 
convention. The expression in (\ref{eqn-anf}) is called 
the \textit{algebraic normal form} of $f$. The expression $ \prod_{j=1}^d x_{i_j}$ is called a term of 
degree $d$ in the algebraic normal form, which appears in the form if and only if its coefficient is $1$.  

We will need the following lemma when we study the algebraic normal forms of the characteristic function $f_{\bD}$ of $t$-designs $\bD$ later \cite[p. 15]{BJL}.  

\begin{lemma}\label{lem-Stinson} 
Suppose that $(\cP, \cB)$ is a $t$-$(n, k, \lambda)$ design. Suppose that $Y \subseteq \cP$, 
where $|Y| = s \leq t$. Then there are exactly $\lambda_s$ blocks in $\cB$ that contain all 
the points in $Y$, where the $\lambda_s$ is defined in (\ref{eqn-fundms}).  
\end{lemma}

\begin{theorem}\label{thm-anmprop}
Let $\bD=([1..n], \cB)$ be a $t$-$(n, k, \lambda)$ design, where $n \geq k \geq t \geq 1$, and let $f_{\bD}$ be the characteristic 
function of $\bD$. Then we have the following regarding the algebraic normal form of $f_{\bD}$: 
\begin{itemize}
\item All terms of degree no more than $k-1$ vanish. 
\item A term $x_{i_1} x_{i_2} \cdots x_{i_k}$ appears if and only if $\{i_1, i_2, \ldots, i_k\}$ is 
      a block in $\cB$, where $\{i_1, i_2, \ldots, i_k\}$ is a $k$-subset of $[1..n]$. Hence, there are exactly $b$ terms of degree $k$ in the algebraic normal form.  
\item For each $h$ with $1 \leq h \leq t$, either all terms of degree $n-h$ appear or none of them 
      appears, depending on the parity of $\overline{\lambda_h}$.        
\item The term $x_1x_2 \cdots x_n$ of degree $n$ appears if and only if $b$ is odd.  
\end{itemize}
\end{theorem} 

\begin{proof}
Let $\bD=(\cP, \cB)$, where $\cB=\{B_1, B_2, \ldots, B_b\}$. Let 
$$ 
\varphi^{-1}(B_i)=(b_{i,1}, b_{i,2}, \ldots, b_{i, n}) \in \gf(2)^n 
$$ 
for $1 \leq i \leq b$. Let $\bar{b}_{i,j}=b_{i,j}+1$ for all $i$ and $j$. 
Denote by $B_{(d, n)}$ the set of all $d$-subsets of $[1..n]$.  

It follows from (\ref{eqn-anf}) that  
\begin{eqnarray}
f_{\bD}(x) &=& \sum_{i=1}^b \prod_{j=1}^n \bar{b}_{i,j} + b \prod_{i=1}^n x_i \nonumber \\
& & + 
\sum_{d=1}^{n-1} \sum_{\{i_1, \ldots, i_d\} \in B_{(d, n)}} \left( \sum_{i=1}^b \prod_{j \in [1..n]  \setminus \{i_1, \ldots, i_d\}} \bar{b}_{i,j} \right) \prod_{j=1}^d x_{i_j}.    
\end{eqnarray}
Since $k \geq 1$, for each fixed $i$ one of $\bar{b}_{i,j}$ must be zero. Consequently, the constant 
term 
$$ 
 \sum_{i=1}^b \prod_{j=1}^n \bar{b}_{i,j}=0. 
$$ 
Note that the coefficient of the term $\prod_{h=1}^d x_{i_h}$ is 
\begin{eqnarray}\label{eqn-coeffsumm}
\sum_{i=1}^b \prod_{j \in [1..n] \setminus \{i_1, \ldots, i_d\}} \bar{b}_{i,j}. 
\end{eqnarray}
Consider now the case that $1 \leq d \leq k-1$. In this case, $n -d > n-k$. It then follows that 
$$ 
\prod_{j \in [1..n] \setminus \{i_1, \ldots, i_d\}} \bar{b}_{i,j}=0 
$$ 
for each $i$ with $1 \leq i \leq b$. We then deduce that the sum in (\ref{eqn-coeffsumm}) is zero. 
This completes the proof of the conclusion in the first part. 

We now prove the conclusion of the second part. Consider any $k$-subset $\{i_1, i_2, \ldots, i_k\}$ 
of $[1..n]$ and the corresponding product $x_{i_1} x_{i_2} \cdots x_{i_k}$ whose 
coefficient is 
\begin{eqnarray}\label{eqn-termkkk}
\sum_{i=1}^b \bar{b}_{i, j_1} \bar{b}_{i, j_2} \cdots \bar{b}_{i, j_{n-k}},  
\end{eqnarray} 
where $\{j_1, j_2, \ldots, j_{n-k}\}=[1..n]  \setminus \{i_1, i_2, \ldots, i_k\}$. 
Since $\bD$ is simple, the summation in (\ref{eqn-termkkk}) is $1$ if and only if 
for exactly one $i$ with $1 \leq i \leq b$ the vector $(b_{i, j_1}, b_{i, j_2}, \ldots, b_{i, j_{n-k}})$ 
is the all-zero vector, which is the same as that the vector $(b_{i, i_1}, b_{i, i_2}, \ldots, b_{i, i_{k}})$ 
is the all-one vector. The desired conclusion in the second part then follows.  

We then prove the conclusion in the third part. Let $1 \leq h \leq t$. The coefficient of the term 
$x_{i_1} x_{i_2} \cdots x_{i_{n-h}}$ is 
$$ 
\sum_{i=1}^b \prod_{u=1}^h \bar{b}_{i, j_u}, 
$$ 
where $\{j_1, j_2, \ldots, j_h\}=[1..n]  \setminus \{i_1, i_2, \ldots, i_{n-h}\}$. 
By Lemma \ref{lem-complementdesign}, the total number of $\varphi^{-1}(B_i)$ such that $\prod_{u=1}^h \bar{b}_{i, j_u}=1$ 
is equal to 
$ 
\overline{\lambda_h},
$  
which depends on $h$ and is independent of the specific elements in $\{j_1, j_2, \ldots, j_h\}$. 
Hence, the conclusion of the third part follows.  
The last conclusion is obvious. 
\end{proof} 

Note that Theorem \ref{thm-anmprop} does not give information on terms of degree between $k+1$ and 
$n-t-1$ in the algebraic normal form of $f_{\bD}(x)$ of a $t$-design $\bD$. In Example \ref{exam-fanofunc}, 
Theorem \ref{thm-anmprop} gives information on all terms of degree in $\{0,1,2,3, 5,6,7\}$, but not terms 
of degree $4$. In fact, in the algebraic normal form in Example \ref{exam-fanofunc} only $28$ out of $35$ 
terms of degree $4$ appear.

\section{Properties of the spectra of the characteristic function $f_{\bD}$ of $t$-designs} 

Our task in this section is to provide further information on the spectra of the characteristic 
function $f_{\bD}$ of $t$-designs, in addition to the information given in Theorem \ref{thm-maincharacterisationthm}. Such information may be useful in settling the existence of certain $t$-designs.

The following lemma will be employed later in this paper, and can be proved easily.  

\begin{lemma}\label{lem-WalshProperties} 
Let $f(x)$ be a Boolean function with $n$ variables. Then 
\begin{enumerate}
\item $\sum_{w \in \gf(2)^n} \hat{f}(w) =2^n f(0)$; and 
\item $\sum_{w \in \gf(2)^n} \hat{f}(w)^2 = 2^{n} \sum_{z \in \gf(2)^n} f(z) =2^n \wt(f)$.
\end{enumerate} 
\end{lemma}

\begin{lemma}\label{lem-complementaryWalsh} 
Let $\bD=(\cP, \cB)$ be an incidence structure, where the point set $\cP=[1..n]$,  
the block set $\cB=\{B_1, B_2, \ldots, B_b\}$, the block size $|B_i|$ is $k$, and $k$ and $b$ 
are positive integers. Let $f_{\bD}$ be the characteristic function of $\bD$. Then 
\begin{eqnarray}
\hat{f}_{\bD}(\bar{w})=(-1)^k \hat{f}_{\bD}(w), 
\end{eqnarray}
where $w \in \gf(2)^n$ and $\bar{w}=\bone + w$ which is the complement of $w$. 
\end{lemma} 

\begin{proof}
Note that $|B_i|=k$ for each $i$ with $1 \leq i \leq b$. By definition, we have 
\begin{eqnarray*}
\hat{f}_{\bD}(\bar{w})=\sum_{i=1}^b (-1)^{(\bone +w) \cdot \varphi^{-1}(B_i)} = (-1)^k \sum_{i=1}^b (-1)^{w \cdot \varphi^{-1}(B_i)} 
= (-1)^k \hat{f}_{\bD}(w). 
\end{eqnarray*}
\end{proof}

\begin{theorem}\label{thm-zerospectrum}
Let $\bD=(\cP, \cB)$ be a $t$-$(n, n/2, \lambda)$ design, where $n$ is even. Then $\hat{f}_{\bD}(w)=0$ for all $w \in \gf(2)^n$ with 
$\wt(w)$ being odd and $1 \leq \wt(w) \leq t$.  
\end{theorem}

\begin{proof}
By the second part of Theorem \ref{thm-ms21}, $P_{n/2}(i)=0$ for all odd $i$ with $0 \leq i \leq n$. The desired conclusion then 
follows from Theorem \ref{thm-maincharacterisationthm}.  
\end{proof}

\section{The spectra of the characteristic function $f_{\bD}(x)$ of $\frac{n-2}{2}$-$\left(n, \frac{n}{2}, 1\right)$ designs}\label{sec-SpectraSteinerAll} 

In this section, we determine the spectra of the characteristic 
function $f_{\bD}$ of $\frac{n-2}{2}$-$\left(n, \frac{n}{2}, 1\right)$ designs.

\subsection{Necessary conditions for the existence of a $t$-$(n, k, \lambda)$ design}

As a corollary of Lemma \ref{lem-designbasicp}, we have the following. 

\begin{corollary}\label{cor-divisibility} 
If a $t$-$(n, k, \lambda)$ design exists, then 
\begin{eqnarray}\label{eqn-tdesignnecessty}
\binom{k-i}{t-i} \mbox{ divides } \lambda \binom{n-i}{t-i}
\end{eqnarray}
for all integer $i$ with $0 \leq i \leq t$. 
\end{corollary} 

As a corollary of Theorem \ref{thm-maincharacterisationthm}, we have also the following. 

\begin{corollary}\label{cor-divisibility2}  
If a $t$-$(n, k, \lambda)$ design exists, then 
\begin{eqnarray}\label{eqn-tdesignnecessty2}
\binom{n-t}{k-t} \mbox{ divides } \lambda P_k(h)
\end{eqnarray}
for all integer $h$ with $0 \leq h \leq t$. 
\end{corollary}

Note that the divisibility conditions in (\ref{eqn-tdesignnecessty}) should be equivalent to 
those in (\ref{eqn-tdesignnecessty2}) if a $t$-$(n, k, \lambda)$ exists. It is open if they are 
equivalent.

The next result is a special case of Corollary \ref{cor-divisibility} \cite[p. 102]{HBComb}, 
and is equivalent to the conditions in Corollary \ref{cor-divisibility}.  

\begin{corollary}\label{cor-divisibilitySpecial} 
If a $t$-$(n, t+1, 1)$ design exists, then 
\begin{eqnarray}\label{eqn-tdesignnecesstySpecial}
\gcd(n-t, \lcm(1,2, \ldots, t+1))=1. 
\end{eqnarray} 
\end{corollary}

The following follows from Corollary \ref{cor-divisibilitySpecial}.  

\begin{theorem}\label{thm-fourthfactor}
If an $(n-2)/2$-$(n, n/2, 1)$ design exists for even $n \geq 4$, then $n \equiv 0 \pmod{4}$ and 
$(n+2)/2$ is a prime. 
\end{theorem} 

Later in this paper we will make use of the fact that $n \equiv 0 \pmod{4}$ from time to time. 
The next two theorems are from \cite[p. 102]{HBComb}, and document some necessary conditions of 
the existence of Steiner systems. These bounds are derived from the Johnson bounds for constant 
weight codes. 

\begin{theorem}
If a $t$-$(n, k, 1)$ design exists, then 
\begin{eqnarray}\label{eqn-filter2}
\binom{k}{t-1} \frac{k-t}{n-k-1} \leq 
\left\lfloor \frac{k}{t-1} \left\lfloor \frac{k-1}{t-2} \left\lfloor \cdots \left\lfloor  \frac{k-t+3}{2}    \right\rfloor \cdots \right\rfloor \right\rfloor \right\rfloor 
\end{eqnarray} 
and 
\begin{eqnarray}\label{eqn-filter3}
\binom{k}{k-t+1} \frac{k-t}{n-k-1} \leq 
\left\lfloor \frac{k}{k-t+1} \left\lfloor \frac{k-1}{k-t} \left\lfloor \cdots \left\lfloor  \frac{t+1}{2}    \right\rfloor \cdots \right\rfloor \right\rfloor \right\rfloor.  
\end{eqnarray} 
\end{theorem}
 
\begin{theorem}
Let $t=2h+\delta$ with $\delta \in \{0, 1\}$. If a $t$-$(n, k, 1)$ design exists, then 
\begin{eqnarray}\label{eqn-filter4}
\binom{n}{t} \geq \left( \frac{n}{k}\right)^{\delta} \binom{n-\delta}{h} \binom{k}{t}. 
\end{eqnarray}
\end{theorem}

The following result is a fundamental result whose proof can be found in \cite[p. 103]{BJL}. 

\begin{theorem}\label{thm-trivialdesigns}
Every $t$-$(n, k, \lambda)$ design with $n \leq k +t$ is trivial in the sense that all $k$-subsets 
occur as blocks. 
\end{theorem}

\begin{theorem}\label{thm-mytrivialdesign}
The only $t$-$(n, k, \lambda)$ design with $t \geq n/2$ is the trivial $t$-$(n, n, 1)$ design $(\cP, \cB)$ 
with $\cP=[1..n]$ and $\cB=\{[1..n]\}$. 
\end{theorem} 

\begin{proof}
Suppose $\bD=(\cP, \cB)$ is a $t$-$(n, k, \lambda)$ design with $t \geq \lfloor n/2 \rfloor$. Then by Theorem 
\ref{thm-maincharacterisationthm} and Lemma \ref{lem-complementaryWalsh}, the Walsh spectra of 
$f_{\bD}$ is unique. Hence, the Boolean function $f_{\bD}$ is uniquely determined. This means that it must 
be the trivial $t$-$(n, n, 1)$ design $(\cP, \cB)$ 
with $\cP=[1..n]$ and $\cB=\{[1..n]\}$.   
\end{proof}

We remark that the conclusion of Theorem \ref{thm-mytrivialdesign} is stronger than that of Theorem 
\ref{thm-trivialdesigns} in this special case. 

Note that for $t$-$(n, k, \lambda)$ designs, we have $1 \leq t \leq k \leq n$. In view of Theorems 
\ref{thm-trivialdesigns} and \ref{thm-mytrivialdesign}, the most interesting designs are 
$(n-2)/2$-$(n, n/2, \lambda)$ designs for even $n$ and $(n-3)/2$-$(n, k, \lambda)$ 
designs for odd $n$ and $k \in \{(n-1)/2, (n+1)/2\}$.

\subsection{The spectra of the characteristic function $f_{\bD}(x)$ of $(n-2)/2$-$(n, n/2, 1)$  
Steiner systems}\label{sec-may27} 

Theorem \ref{thm-maincharacterisationthm} and Lemma \ref{lem-complementaryWalsh} show that $\hat{f}_{\bD}(w)$ 
is known for all $w \in \gf(2)^n$ except those with $\wt(w)=n/2$ when $\bD$ is an $(n-2)/2$-$(n, n/2, 1)$  
Steiner system. In this section, we determine $\hat{f}_{\bD}(w)$ for all $w \in \gf(2)^n$ with 
$\wt(w)=n/2$. By Theorem \ref{thm-fourthfactor}, $n=2p-2$ for a prime $p \geq 3$.   

\begin{theorem}
Let $\bD=(\cP, \cB)$ be an incidence structure, where the point set $\cP=[1..n]$,  
the block set $\cB=\{B_1, B_2, \ldots, B_b\}$, the block size $|B_i|$ is $k$, and $k$ and $b$ 
are positive integers. If $\bD$ is an $(n-2)/2$-$(n, n/2, 1)$ design, where $n$ is even, then 
\begin{eqnarray}\label{eqn-1str}
\sum_{\myatop{w \in \gf(2)^n}{\wt(w)=n/2}} \hat{f}_{\bD}(w)  = 
- \frac{4}{n+2} \sum_{h=0}^{(n-2)/2} \binom{n}{h} P_{n/2}(h), 
\end{eqnarray}
and
\begin{eqnarray}\label{eqn-2str}
\sum_{\myatop{w \in \gf(2)^n}{\wt(w)=n/2}} \left(\hat{f}_{\bD}(w)\right)^2 = 
2^{n} b -  \frac{8}{(n+2)^2}   \sum_{h=0}^{(n-2)/2} \binom{n}{h} \left(P_{n/2}(h)\right)^2.
\end{eqnarray} 
\end{theorem} 

\begin{proof} 
By Theorem \ref{thm-fourthfactor}, the block size $k=n/2$ is even. 
We first prove (\ref{eqn-1str}). Note that $f_{\bD}(\bzero)=0$, as $k=n/2 \geq 2$. It follows from Lemmas \ref{lem-WalshProperties}, 
\ref{lem-complementaryWalsh},  and 
Theorem \ref{thm-maincharacterisationthm} that 
\begin{eqnarray*}
\sum_{\myatop{w \in \gf(2)^n}{\wt(w)=n/2}} \hat{f}_{\bD}(w)  
&=& 2^n f(\bzero) - 
\left((-1)^{k}+1\right) \sum_{\myatop{w \in \gf(2)^n}{0 \leq \wt(w) \leq \frac{n-2}{2}}} \hat{f}_{\bD}(w) \\
&=& -2 \left( \sum_{\myatop{w \in \gf(2)^n}{0 \leq \wt(w) \leq \frac{n-2}{2}}} \hat{f}_{\bD}(w) \right) \\ 
&=& - \frac{4}{n+2} \sum_{h=0}^{(n-2)/2} \binom{n}{h} P_{n/2}(h).  
\end{eqnarray*} 

We now prove (\ref{eqn-2str}). It follows from Lemmas \ref{lem-WalshProperties}, 
\ref{lem-complementaryWalsh},  and 
Theorem \ref{thm-maincharacterisationthm} that 
\begin{eqnarray*}
\sum_{\myatop{w \in \gf(2)^n}{\wt(w)=n/2}} \left(\hat{f}_{\bD}(w)\right)^2   
&=& 2^{n} b - 
 2\sum_{\myatop{w \in \gf(2)^n}{0 \leq \wt(w) \leq \frac{n-2}{2}}} \left(\hat{f}_{\bD}(w)\right)^2  \\
&=& 2^{n} b -  \frac{8}{(n+2)^2}   \sum_{h=0}^{(n-2)/2} \binom{n}{h} \left(P_{n/2}(h)\right)^2. 
\end{eqnarray*} 
\end{proof}

\begin{theorem}
Let $\bD=(\cP, \cB)$ be an $(n-2)/2$-$(n, n/2, 1)$ design, where $n \equiv 0 \pmod{4}$. If $w \in \gf(2)^n$ has odd weight $h$ with 
$1 \leq h \leq t$, then $\hat{f}_{\bD}(w)=0$.   

\end{theorem} 

\begin{proof}
The desired conclusion follows from Theorem \ref{thm-zerospectrum}. 
\end{proof}

\begin{theorem}\label{thm-SteinerSpect1}
Let $w \in \gf(2)^n$ with $\support(w)=B_i$ being a block of an $(n-2)/2$-$(n, n/2, 1)$ design 
$\bD=(\cP, \cB)$. Then 
\begin{eqnarray}\label{eqn-yyyi}
\hat{f}_{\bD}(w) = 2^{n/2} - \frac{2\sum_{h=1}^{n/2} \binom{n/2}{h} \sum_{\ell=0}^{h-1} (-1)^\ell \binom{(n+2)/2}{\ell +1}}{n+2}.  
\end{eqnarray}
\end{theorem} 

\begin{proof} 
Note that $n \equiv 0 \pmod{4}$. 
It follows from 
Theorem \ref{thm-mymay22} that 
$$ 
\lambda_{(\frac{n}{2}-j, j)}=(-1)^j + \frac{2 (-1)^{j-1} \sum_{\ell=0}^{j-1} (-1)^\ell \binom{(n+2)/2}{\ell +1}}{n+2}   
$$ 
for $1 \leq j \leq n/2$. 
We have then 
\begin{eqnarray*}
\hat{f}_{\bD}(w) 
&=& \sum_{j=1}^b (-1)^{|B_j \cap B_i|} \\
&=& 1+ \sum_{h=0}^{(n-2)/2} (-1)^h \binom{n/2}{h} \lambda_{(h, \frac{n}{2}-h)} \\
&=& 1+ \sum_{h=1}^{n/2} (-1)^h \binom{n/2}{h} \lambda_{(\frac{n}{2}-h, h)} \\
&=& 1+ \sum_{h=1}^{n/2} (-1)^h \binom{n/2}{h} \left[ (-1)^h + \frac{2(-1)^{h-1} \sum_{\ell=0}^{h-1} (-1)^\ell \binom{(n+2)/2}{\ell +1} }{n+2} \right] \\ 
&=& \sum_{h=0}^{n/2} \binom{n/2}{h} - \frac{2\sum_{h=1}^{n/2} \binom{n/2}{h} \sum_{\ell=0}^{h-1} (-1)^\ell \binom{(n+2)/2}{\ell +1}}{n+2}.
\end{eqnarray*} 
The proof is then completed. 
\end{proof}

The following theorem will complete the task of determining the spectra of the characteristic function 
$f_{\bD}$ for $(n-2)/2$-$(n, n/2, 1)$ Steiner systems. 

\begin{theorem}\label{thm-Steinerspect2}
Let $\bD=([1..n], \cB)$ be an $(n-2)/2$-$(n, n/2, 1)$ Steiner system. Let $w \in \gf(2)^n$ with 
$\wt(w)=n/2$. Let $B=\support(w)$. Then 
$$ 
\hat{f}_{\bD}(w)=\sum_{i=0}^{n/2} (-1)^i y_i, 
$$ 
where $y_0, y_1, \ldots, y_{n/2}$ are uniquely determined by the following system of equations: 
\begin{eqnarray}\label{eqn-final28}
\left\{ 
\begin{array}{l}
\sum_{i=r}^{n/2} \binom{i}{r} y_i = \binom{n/2}{r} \lambda_r, \ \ \ 0 \leq r \leq \frac{n-2}{2}, \\ 
y_0=y_{\frac{n}{2}} = 
    \left\{ 
    \begin{array}{ll}
    1 & \mbox{ if } B \in \cB, \\
    0 & \mbox{ if } B \not\in \cB. 
    \end{array}
\right.   
\end{array}
\right. 
\end{eqnarray}
\end{theorem}

\begin{proof}
Define 
$$ 
y_i=|\{1 \leq j \leq b: |B \cap B_j|=i\}|
$$ 
for $0 \leq i \leq n/2$. It then follows from \cite[p. 179]{Alltop} that 
\begin{eqnarray*}
\sum_{i=r}^{n/2} \binom{i}{r} y_i = \binom{n/2}{r} \lambda_r, \ \ \ 0 \leq r \leq \frac{n-2}{2}   
\end{eqnarray*} 
and 
\begin{eqnarray*}
y_0 - y_{n/2}=\sum_{r=0}^{\frac{n-2}{2}} (-1)^r \binom{n/2}{r} \lambda_r. 
\end{eqnarray*}
One can prove that 
$$ 
\sum_{r=0}^{\frac{n-2}{2}} (-1)^r \binom{n/2}{r} \lambda_r=0. 
$$ 
The desired conclusion then follows from 
$$
\hat{f}_{\bD}(w)=\sum_{j=1}^b (-1)^{|B \cap B_j|} =\sum_{i=0}^{n/2} (-1)^i y_i. 
$$
\end{proof}

We remark that the values $y_0, y_1, \ldots, y_{n/2}$ in Theorem \ref{thm-Steinerspect2} can be 
derived easily from (\ref{eqn-yyyi}), though their expressions may look a little complex. As a consequence of Theorem \ref{thm-Steinerspect2}, we have the following, 

\begin{corollary}\label{cor-Steinerselfcompl} 
Every $(n-2)/2$-$(n, n/2, 1)$ design $\bD$ is self-complementary, i.e., the complement of a block is also 
a block of the design, i.e., $\overline{\bD}=\bD$. 
\end{corollary}  

\begin{proof}
The desired conclusion follows from the fact that $y_0=y_{n/2}$ in the proof of Theorem \ref{thm-Steinerspect2}. 
\end{proof}

Theorem \ref{thm-Steinerspect2} and Corollary \ref{cor-Steinerselfcompl} tell us that $\hat{f}_{\bD}(w)$ 
takes on two different values depending on whether $\support(w) \in \cB$ or  $\support(w) \in \binom{[1..n]}{n/2} \setminus \cB$ for all $w \in \gf(2)^n$ with $\wt(w)=n/2$.  

\begin{table}[hb]
\center 
\caption{Spectra of $f_{\bD}$}\label{tab-spectraStein}
{
\begin{tabular}{lr}
\hline
Weight of $w$    & Multiset $\{ \hat{f}_{\bD}(w) \}$  \\ \hline
$0, 12$          & $\{ 132 \}$ \\
$1,11$          & $\{ 0^{12} \}$ \\ 
$2,10$          & $\{ -12^{66} \}$ \\  
$3,9$           & $\{ 0^{220} \}$ \\  
$4,8$           & $\{ 4^{495} \}$ \\  
$5,7$           & $\{ 0^{792} \}$ \\  
$6$             & $\{ -12^{792}, 52^{132}\}$ \\  
\hline 
\end{tabular}
}
\end{table}

\begin{example}\label{exam-SteinerSys} 
Consider the $5$-$(12, 6, 1)$ Steiner system from the extended ternary Golay code of length $12$. The spectrum $\hat{f}_{\bD}(w)$ 
is given in Table \ref{tab-spectraStein}. 
\end{example} 

Theorem \ref{thm-Steinerspect2} does not give an explicit expression of $\hat{f}_{\bD}(w)$ when 
$|\support(w)|=n/2$ and $\support(w)$ is not a block of $\bD$. We would now give an explicit 
expression of $\hat{f}_{\bD}(w)$ for this case. 

As before, let $\bD=([1..n], \cB)$ be an $(n-2)/2$-$(n, n/2, 1)$ design. Define  
$\tilde{\cB}=\binom{[1..n]}{n/2} \setminus \cB$.  Theorems \ref{thm-SteinerSpect1} 
and \ref{thm-Steinerspect2} show that 
\begin{eqnarray}\label{eqn-july161}
\hat{f}_{\bD}(w) = 
\left\{ 
\begin{array}{ll}
a & \mbox{ if } w \in \varphi^{-1}(\cB), \\ 
\tilde{a} & \mbox{ if } w \in \varphi^{-1}(\tilde{\cB}), 
\end{array}
\right. 
\end{eqnarray} 
where $a$ is the number of the right-hand side of (\ref{eqn-yyyi}), and $\tilde{a}$ is implied in 
Theorem \ref{thm-Steinerspect2}. We now determine $\tilde{a}$ specifically.  

Define $\Delta=2/(n+2)$. 
By Theorem \ref{thm-july161}, 
\begin{eqnarray}\label{eqn-july162}
P_{n/2}(h) = (-1)^{n/2} P_{n/2}(n-h) = P_{n/2}(n-h),  
\end{eqnarray} 
where $0 \leq h \leq n$. 
By Lemma \ref{lem-complementaryWalsh}
\begin{eqnarray}\label{eqn-july163}
\hat{f}_{\bD}(\bar{w})=(-1)^{n/2} \hat{f}_{\bD}(w)= \hat{f}_{\bD}(w), 
\end{eqnarray}
where $w \in \gf(2)^n$ and $\bar{w}=\bone + w$ which is the complement of $w$. 

By (\ref{eqn-101010}), we deduce that 
\begin{eqnarray}\label{eqn-july165}
2^nf_{\bD}(x) 
&=& \sum_{\wt(w) \neq n/2} \hat{f}_{\bD} (w) (-1)^{w\cdot x} + 
   \sum_{\wt(w) = n/2} \hat{f}_{\bD} (w) (-1)^{w\cdot x} \nonumber \\
&=& \Delta \sum_{h=0}^n P_{n/2}(h) \sum_{\wt(w)=h} (-1)^{w\cdot x} \mbox{ \hfill (by (\ref{eqn-july162}) and (\ref{eqn-july163})) } \nonumber \\
& & - \Delta P_{n/2}(n/2) \sum_{\wt(w)=n/2} (-1)^{w\cdot x} 
    + \sum_{\wt(w) = n/2} \hat{f}_{\bD} (w) (-1)^{w\cdot x} \nonumber \\  
&=& \Delta \sum_{h=0}^n P_{n/2}(h) P_{h}(\wt(x)) \mbox{ \hfill (by Theorem \ref{thm-july163}) } \nonumber \\
& & - \Delta P_{n/2}(n/2) P_{n/2}(\wt(x))  
    + \sum_{\wt(w) = n/2} \hat{f}_{\bD} (w) (-1)^{w\cdot x} \nonumber \\  
&=& 2^n \Delta \delta_{n/2, \wt(x)} \mbox{ \hfill (by Theorem \ref{thm-july164}) } \nonumber \\
& & - \Delta P_{n/2}(n/2) P_{n/2}(\wt(x))  
    + \sum_{\wt(w) = n/2} \hat{f}_{\bD} (w) (-1)^{w\cdot x}.   
\end{eqnarray} 

It follows from (\ref{eqn-july161}) that 
\begin{eqnarray}\label{eqn-july166}
\sum_{\wt(w) = n/2} \hat{f}_{\bD} (w) (-1)^{w\cdot x} 
&=& \sum_{w \in \varphi^{-1}(\cB)} a (-1)^{w\cdot x} + \sum_{w \in \varphi^{-1}(\tilde{\cB})} \tilde{a} (-1)^{w\cdot x} \nonumber \\ 
&=& \sum_{w \in \varphi^{-1}(\cB)} (a-\tilde{a}) (-1)^{w\cdot x} + \sum_{\wt(w)=n/2} \tilde{a} (-1)^{w\cdot x} \nonumber \\ 
&=&  (a-\tilde{a})\hat{f}_{\bD}(x) +  \tilde{a} P_{n/2}(\wt(x)).  
\end{eqnarray} 

Combining (\ref{eqn-july165}) and (\ref{eqn-july166}) yields 
\begin{eqnarray}\label{eqn-july167}
2^n f_{\bD}(x)=2^n \Delta \delta_{n/2, \wt(x)} + \left(\tilde{a}-\Delta P_{n/2}(n/2)\right) P_{n/2}(\wt(x)) 
+ (a-\tilde{a})\hat{f}_{\bD}(x).    
\end{eqnarray}  
Consequently, when $x=\bzero$, 
\begin{eqnarray}\label{eqn-july168}
\left(\tilde{a}-\Delta P_{n/2}(n/2)\right) P_{n/2}(0) 
+ (a-\tilde{a})\hat{f}_{\bD}(\bzero)=0.    
\end{eqnarray}  
By Theorem \ref{thm-ms21}, $P_{n/2}(0)=\binom{n}{n/2}$. By definition, 
$$ 
\hat{f}_{\bD}(\bzero)=\lambda_0=b=\frac{\binom{n}{(n-2)/2}}{\binom{n/2}{1}}. 
$$ 
Solving this equation gives 
$$ 
\tilde{a}=-\frac{2(a - P_{n/2}(n/2))}{n}. 
$$
This proves the following theorem. 

\begin{theorem}\label{thm-SteinerSpect3}
Let $w \in \gf(2)^n$ be such that $|\support(w)|=n/2$ and $\support(w)$ is not a block of an $(n-2)/2$-$(n, n/2, 1)$ design 
$\bD=(\cP, \cB)$. Then 
\begin{eqnarray}\label{eqn-yyyiii} 
\hat{f}_{\bD}(w) = -\frac{2(a - P_{n/2}(n/2))}{n}, 
\end{eqnarray}
where 
\begin{eqnarray*}
a = 2^{n/2} - \frac{2\sum_{h=1}^{n/2} \binom{n/2}{h} \sum_{\ell=0}^{h-1} (-1)^\ell \binom{(n+2)/2}{\ell +1}}{n+2}.  
\end{eqnarray*}
\end{theorem}

\subsection{The existence of Steiner systems $(n-2)/2$-$(n, n/2, 1)$}

We are concerned with the existence of $(n-2)/2$-$(n, n/2, 1)$ designs 
for even $n$. The integers $n$ in the range $8 \leq n \leq 150$ that satisfies the  
conditions in (\ref{eqn-tdesignnecessty}), (\ref{eqn-tdesignnecessty2}) and 
(\ref{eqn-tdesignnecesstySpecial}) are given in the set 
\begin{eqnarray}
\{8, 12, 20, 24, 32, 36, 44, 56, 60, 72, 80, 84, 92, 104, 116, 120, 132, 140,
144\}. 
\end{eqnarray} 
The parameters $(n-2)/2$-$(n, n/2, 1)$ for all the $n$ in the set above also satisfy the conditions in  
(\ref{eqn-filter2}), (\ref{eqn-filter3}), and (\ref{eqn-filter4}). 
So, they are admissible parameters of $(n-2)/2$-$(n, n/2, 1)$ designs  
according to these known necessary conditions.  

Experimental data indicates that there are infinitely many admissible parameters $(n-2)/2$-$(n, n/2, 1)$. 
Steiner systems with parameters $3$-$(8,4,1)$ and $5$-$(12,6,1)$ do exist. 

\subsection{The construction of Steiner systems $(n-2)/2$-$(n, n/2, 1)$} 

The correspondence from a Boolean function $f(x)$ to its spectra is not one-to-one. For the 
characteristic function $f_{\bD}(x)$ of an $(n-2)/2$-$(n, n/2, 1)$ Steiner system, $\hat{f}_{\bD}(w)$ 
is a constant for all $w \in \gf(2)^n$ with fixed weight $h$ except $h=n/2$. Since 
$\hat{f}_{\bD}(w)$ takes on two distinct values for all $w \in \gf(2)^n$ with $\wt(w)=n/2$, 
the spectra of an $(n-2)/2$-$(n, n/2, 1)$ Steiner system does not give enough information for 
constructing the characteristic function of such Steiner system with the inverse Walsh transform 
approach.

\section{Binary linear codes from the characteristic functions of $t$-designs} 

The incidence matrix of a $t$-$(n, k, \lambda)$ design $\bD$ can be viewed as a matrix over any field 
$\gf(q)$ and its rows span a linear code of length $n$ over $\gf(q)$. This is the classical construction 
of linear codes from $t$-designs and has been intensively studied \cite{AK92}. 

Any $t$-$(n, k, \lambda)$ design $\bD$ can also be employed to construct a binary linear code of length 
$2^n-1$ and dimension $n+1$. This is done via the characteristic Boolean function of the design. It is  
likely that the weight distribution of the code could be determined. Below we demonstrate this approach 
with $(n-2)/2$-$(n, n/2, 1)$ Steiner systems. 

Let $f(x)$ be a Boolean function with $n$ variables such that $f(\bzero)=0$ but $f(x)=1$ for at least 
one $x \in \gf(2)^n$. We now define a linear code by 
\begin{eqnarray}\label{eqn-mycode}
\C_f=\{(uf(x)+v \cdot x)_{x \in \gf(2)^n \setminus \{\bzero\}}: u \in \gf(2), \ v \in \gf(2)^n\}. 
\end{eqnarray}    
This construction goes back to \cite{CH17, CD07, Japanese}. 

The following theorem should be well known. However, for completeness we will sketch a proof for it. 

\begin{theorem}\label{thm-codeparam}
The binary code $\C_f$ in (\ref{eqn-mycode}) has length $2^n-1$ and dimension $n+1$. In addition, 
the weight distribution of $\C_f$ is given by the following multiset union: 
\begin{eqnarray*}
\{ 2^{n-1}+ \hat{f}(w): w \in \gf(2)^n \setminus \{\bzero\} \} \cup \{\hat{f}(0)\} 
\cup \{2^{n-1}: w \in \gf(2)^n \setminus \{\bzero\} \} \cup \{0\}. 
\end{eqnarray*}
\end{theorem}

\begin{proof}
It is easily seen that 
\begin{eqnarray*}
\sum_{x \in \gf(2)^n} (-1)^{f(x)+w \cdot x} 
= \left\{ 
\begin{array}{ll}
2^n-2\hat{f}(\bzero) & \mbox{ if } w = \bzero, \\ 
-2\hat{f}(w)         & \mbox{ if } w \neq \bzero. 
\end{array}
\right. 
\end{eqnarray*} 
On the other hand, 
$$ 
\sum_{x \in \gf(2)^n} (-1)^{f(x)+w \cdot x} = 
2^n - 2 |\{x \in \gf(2)^n \setminus \{\bzero\}: f(x)+w\cdot x =1 \}|. 
$$ 
Combining the two equations above yields the desired conclusion on the weight distribution. Since 
$f$ is not the zero function, the dimension of the code $\C_f$ must be $n+1$. 
\end{proof}

For any $(n-2)/2$-$(n, n/2, 1)$ design $\bD$, the spectra of the characteristic function $f_{\bD}$ 
were completely determined in Section \ref{sec-SpectraSteinerAll}. Hence, one can write out the weight 
distribution of the binary linear code $\C_{f_{\bD}}$ with the help of Theorem \ref{thm-codeparam}.

\begin{example} 
Let $f_{\bD}$ be the characteristic function of the Steiner system $S(5, 6, 12)$ in Example 
\ref{exam-SteinerSys}. Then the binary code $\C_{f_{\bD}}$ has parameters $[2^{12}-1, 13, 132]$ 
and the weight distribution 
in Table \ref{tab-mycode2}. 
\end{example}    

\begin{table}[ht]
\center 
\caption{Weight distribution}\label{tab-mycode2}
{
\begin{tabular}{lr}
\hline
Weight $w$    & No. of codewords $A_w$  \\ \hline
$0$          & $1$ \\ 
$132$        & $1$ \\
$2^{11}-12$  & $924$ \\
$2^{11}$      & $6143$ \\
$2^{11}+4$  & $990$ \\
$2^{11}+52$  & $132$ \\
$2^{11}+132$  & $1$ \\ 
\hline 
\end{tabular}
}
\end{table}

Another construction of binary linear codes with Boolean functions was treated in \cite{Ding163}. 
After plugging the characteristic function $f_{\bD}$ of any $t$-$(n, k, \lambda)$ design into this 
construction, one 
obtains a binary linear code of length $\lambda\binom{n}{t}/\binom{k}{t}$ and dimension $n$ with 
at most $n+1$ weights.

\section{Conclusions and remarks}

The main contribution of this paper is the spectral characterisation of $t$-designs documented in 
Theorem \ref{thm-maincharacterisationthm}. It is open how to use this characterisation to construct 
or show the existence of $t$-designs with certain parameters. It might be possible to show the 
nonexistence of certain $t$-designs with this characterisation. The second contribution  
is the new necessary condition for the existence of $t$-$(n, k, \lambda)$ designs given in Corollary 
\ref{cor-divisibility2}. The third contribution is the results of the algebraic normal form of the 
characteristic function $f_{\bD}(x)$ of $t$-designs summarised in Theorem \ref{thm-anmprop}. 
Another contribution is the self-complementary property of $(n-2)/2$-$(n, n/2, 1)$ Steiner 
systems introduced in Corollary \ref{cor-Steinerselfcompl}.  
The last contribution is the properties of the spectra $\hat{f}_{\bD}(w)$ for Steiner systems 
with parameters $(n-2)/2$-$(n, n/2, 1)$, which was described in Section \ref{sec-may27}. 
The determination of the spectra $\hat{f}_{\bD}(w)$ for a Steiner system  
with parameters $(n-2)/2$-$(n, n/2, 1)$ allows the determination of the weight distributions of 
two binary linear codes constructed from the Steiner system. Hence, we demonstrated at least 
three ways of constructing a linear code with a $t$-design in this paper.     
  
It was conjectured that the divisibility conditions in (\ref{eqn-tdesignnecessty}) are also efficient 
for the existence of $t$-$(n, k, \lambda)$ Steiner systems except a finite number of exceptional $n$ given 
fixed $t$, $k$ and $\lambda$. Earlier progresses on this conjecture were made in \cite{Wilson72a,Wilson72b,Wilson75}. It is open if 
the characterisation in Theorem \ref{thm-maincharacterisationthm} could be employed to attack this 
problem in a different way. 

As justified in \ref{sec-Dels1} and \ref{sec-Dels2}, the spectral characterisation 
in Theorem \ref{thm-maincharacterisationthm} is different from and much simpler than 
the spectral characterisations of Theorem \ref{thm-delstdesign} and Corollary \ref{cor-May231}.   
As made clear in \ref{sec-Dels3}, 
it is impossible for Theorem \ref{thm-maincharacterisationthm} to be a special case of 
Delarte's Assmus-Mattson Theorem (i.e. Theorem \ref{thm-Delsarte'sAM}). In summary, 
there are three spectral characterisations of combinatorial $t$-designs. The characterisation 
of Theorem \ref{thm-maincharacterisationthm} developed in this paper is the simplest 
and does not depend on the theory of association schemes. In addition, this characterisation 
leads to two applications in coding theory.    

\section*{Acknowledgements} 

The authors thank Dr. Shuxing Li for reading an earlier version of this paper and 
providing helpful comments, and Dr. Yan Zhu for providing the reference for Theorem 
\ref{thm-may183} and helpful comments on $T$-designs and relative $t$-designs.  
The research of Cunsheng Ding was supported by the Hong Kong Research Grants Council, 
under Project No. 16300415.

\appendix 
\section{Spectral extensions of Delsarte's characterisations} 

Delsarte gave a characterisation of combinatorial $t$-designs with $T$-designs in the Johnson scheme and a characterisation of orthogonal arrays with $T$-designs in the Hamming scheme \cite{Dels73}. Delsarte and Seidel gave a characterisation of combinatorial $t$-designs 
with relative $t$-designs \cite{DelsSeid}.  
The objective of this appendix is to extend the three characterisations into spectral ones. The ultimate purpose of this appendix is to show that the spectral characterisation of this paper documented in Theorem \ref{thm-maincharacterisationthm} is different from those of Delsarte and Seidel.      

\subsection{Association schemes} 

An \textit{association scheme with $n$ classes} consists of a finite set $X$ with $v$ points 
together with $n+1$ relations $R_0, R_1, \ldots, R_n$ defined on $X$ which satisfy the following: 
\begin{itemize}
\item[(i)] Each $R_i \subseteq X \times X$ is symmetric, i.e., $(x, y) \in R_i$ implies that 
           $(y, x) \in R_i$. 
\item[(ii)] $R_0, R_1, \ldots, R_n$ form a partition of $X \times X$. 
\item[(iii)] $R_0=\{(x,x): x \in X\}$ is the identity relation.   
\item[(iv)] If $(x, y) \in R_k$, the number of $z \in X$ such that $(x, z) \in R_i$ and 
            $(y, z) \in R_j$ is a constant $c_{ijk}$ depending on $i,j,k$ but not on the 
            particular choice of $x$ and $y$.                   
\end{itemize} 
Let $R=\{R_0, R_1, \ldots, R_n\}$. We call $(X, R)$ an association scheme.  

Let $(X, R)$ be an association scheme, where $R=\{R_0, R_1, \ldots, R_n\}$. 
The \emph{adjacency matrix} $D_i$ of $R_i$ is the $v \times v$ matrix with rows and columns 
labelled by the points of $X$, which is defined by 
\begin{eqnarray*}
(D_i)_{x,y}=\left\{ 
\begin{array}{ll}
1 & \mbox{if $(x, y) \in R_i$,} \\
0 & \mbox{otherwise.} 
\end{array}
\right. 
\end{eqnarray*}   
The Bose-Mesner algebra $\A$ consists of all matrices $\sum_{i=0}^n a_i D_i$, where all $a_i$ 
are real numbers. The Bose-Mesner algebra $\A$ has the following properties \cite[p. 653]{MS77}: 
\begin{itemize}
\item All the matrices in $\A$ are symmetric. 
\item $D_0, D_1, \ldots, D_n$ are linearly independent, and the dimension of $\A$ is $n+1$. 
\item $\A$ has a unique basis of primitive idempotents $J_0, J_1, \ldots, J_n$, which satisfy 
      \begin{eqnarray*}
      J_i^2 &=& J_i, \ i =0, 1, \ldots, n, \\
      J_iJ_k &=& 0, \ i \neq k, \\
      \sum_{i=0}^n J_i &=& I, 
      \end{eqnarray*}       
      where $I$ is the identity matrix.   
\end{itemize}  
Any of the two bases above of the Bose-Mesner algebra $\A$ can be expressed in terms of the other. 
Let 
$$ 
D_k = \sum_{i=0}^n p_k(i) J_i, \ k=0, 1, \ldots, n
$$ 
and 
$$ 
J_k=\frac{1}{v} \sum_{i=0}^n q_k(i) D_i, \ k =0, 1, \ldots, n. 
$$ 
These $p_k(i)$ are the eigenvalues of $D_k$. Let $\rank(J_i)$ be the multiplicity of the 
eigenvalue $p_k(i)$. The matrices $P=[p_i(j)]$ and $Q=[q_i(j)]$ are called the first and  
second eigenmatrices of the scheme. 

An association scheme is called a $P$-polynomial scheme if there exist nonnegative real numbers $z_0=0$, $z_1, \ldots, z_n$ and real polynomials $\Phi_0(z), \Phi_1(z), \ldots, \Phi_k(z)$, 
where $\deg(\Phi_k(z))=k$, such that 
$$ 
p_k(i)=\Phi_k(z_i), \ \ i, k = 0, 1, \ldots, n. 
$$ 
$Q$-polynomial schemes are defined similarly. 

We consider a nonempty subset $C$ of an arbitrary association scheme $(X, R)$ with relations
$R_0,R_1,\ldots,R_n$. The \emph{inner distribution} $B_i$ of $C$ is defined by
\[
B_i=\frac{1}{|C|}|R_i\cap C^2|, \ i=0,1,\ldots,n.
\]
The \emph{outer distribution} $B'_k$ of $C$ is defined by 
\[
B'_k:=\frac{1}{|C|}\sum_{i=0}^n q_k(i)B_i 
\]
for $k=0,1,\ldots,n$, where $q_k(i)$ are entries of the second eigenmatrix $Q$ of $(X, R)$. Delsarte proved that $B'_k \geq 0$ for 
all $k$ with $0 \leq k \leq n$ \cite{Dels73}.

Two useful association schemes are the Hamming scheme and Johnson scheme. Let $X=\gf(2)^n$. Define 
$$ 
R_i=\{(x,y) \in X \times X: \dist(x, y)=i\}
$$ 
for all $i$ with $0 \leq i \leq n$, where $\dist(x, y)$ denotes the Hamming distance between  
$x$ and $y$. It is well known that $(X, R)$ is an association scheme, 
and is called the Hamming scheme (see \cite[p. 665]{MS77} and \cite{Dels73}). The Hamming 
scheme is both a $P$-polynomial and $Q$-polynomial scheme. 

Let $X=\binom{[1..v]}{k}$, which is the set of all $k$-subsets of the set $[1..v]$, 
and where $k \leq v/2$. Define 
$$ 
R_i=\{(x, y) \in X \times X: |x \cap y|=k-i \} 
$$ 
for all $i$ with $0 \leq i \leq k$. It is well known that $(X, \{R_0, \ldots, R_k\})$ 
is an association scheme, and is called the Johnson scheme (see \cite[p. 656]{MS77} and \cite[Section 4.2]{Dels73}).  

Let $(X, R)$ be an association scheme with $n$ classes and let $T$ be any subset of 
$[1..n]$. Let $Y$ be any nonempty subset of $X$. Define 
\begin{eqnarray}\label{eqn-may191}
(BQ)_i:=\sum_{j=0}^n B_j q_i(j), 
\end{eqnarray}
where $B=(B_0, B_1, \ldots, B_n)$ is the inner distribution of $Y$, and $Q=[q_i(j)]$ 
is the second eigenmatrix of the scheme. 
The subset $Y$ of $X$ is called a \textit{$T$-design with respect to $R$} if 
$(BQ)_i=0$ for all $i \in T$ \cite[p. 32]{Dels73}. 

Delsarte gave a characterisation 
of $T$-designs in association schemes \cite[Theorem 3.10]{Dels73}. 
Due to space limitation, we will not document it here. 
In the cases of the Hamming scheme and Johnson scheme, Delsarte's characterisation of 
$T$-designs in association schemes becomes a characterisation of orthogonal arrays 
and combinatorial $t$-designs, respectively. We will introduce and extend them in 
the next two subsections.

\subsection{A spectral extension of Delsarte's characterisation of orthogonal arrays}\label{sec-Dels1} 

An $M \times n$ matrix $A$ with entries from a set of $q$ elements is called an \textit{orthogonal array} of size $M$, $n$ constraints, $q$ levels, strength $k$, and index $\lambda$ if any set 
of $k$ columns of $A$ contains all $q^k$ possible row vectors exactly $\lambda$ times. 
Such an array is denoted by $(M,n, q, k)$. Clearly $M = \lambda q^k$. In this section, 
we consider only the case that $q=2$. 

Now we introduce Delsarte's characterisation of orthogonal arrays for $q=2$, which 
is a special case of his characterisation of $T$-designs in general association schemes 
\cite[Theorem 3.10]{Dels73}. We will follow the refined presentation given in 
\cite[Chapter 21]{MS77}. 

Let $C$ be a subset of $\operatorname{GF}(2)^n$. 
Let $1_C$ denote the characteristic function of $C$, 
which can be viewed as a Boolean function on $\operatorname{GF}(2)^n$.

We denote by $B_i$ the distance distribution (i.e., the inner distribution) of $C$, namely,
\[
B_i =\frac{1}{|C|}|\{(u,v ) \in C \times C : \dist(u, v) = i\}|, 
\]
where $ \dist(u, v)$ denotes the Hamming distance between $u$ and $v$. 
We define the dual distance distribution (i.e., the outer distribution) $B^{'}_i$ by 
\[
B'_k=\frac{1}{|C|}\sum_{i=0}^n P_k(i)B_i \text{ for } k\in[0..n].
\]

The following is Delsarte's characterisation of orthogonal arrays \cite[Theorem 3.10]{Dels73} 
and is a refined version given in \cite[Chapter 21, Theorem 16]{MS77}. 

\begin{theorem}\label{thm-Dels1}
The set $C$ of vectors of $\gf(2)^n$ (viewed as a matrix) is an orthogonal array of size $|C|$, $n$ constraints, $2$ levels, 
strength $t$ and index $|C|/2^t$ if and only if $B'_1=B'_2=\cdots=B'_t=0$.
\end{theorem} 

This characterisation is not a spectral characterisation. Below we extend it into a spectral 
characterisation. To this end, we need the next lemma which should be a known result in the literature. But we provide a different proof below. 

\begin{lemma}\label{lem-kor1}
Let $C$ be a nonempty subset of $\operatorname{GF}(2)^n$. Then
\begin{eqnarray*}
B_k &=& \frac{1}{2^n|C|}\sum_{w\in\operatorname{GF}(2)^n}\hat{1}_C(w)^2P_k(\wt(w)),\\
B'_k &=& \frac{1}{|C|^2}\sum_{w\in{[1..n] \choose k}}\hat{1}_C(w)^2.
\end{eqnarray*}
\end{lemma}

\begin{proof}
We first compute the distance distribution $B_i$ of $C$, and have 
\begin{eqnarray*}
B_i &=& \frac{1}{|C|}\sum_{u,v\in \operatorname{GF}(2)^n}1_C(u)1_C(v)1_{{[1..n]\choose i}}(u+v)\\
&=&\frac{1}{2^n|C|}\sum_{u,v\in \operatorname{GF}(2)^n}1_C(u)1_C(v)
\sum_{w\in\operatorname{GF}(2)^n}\hat{1}_{{[1..n]\choose i}}(w)(-1)^{(u+v)\cdot w}\\
&=&\frac{1}{2^n|C|}\sum_{w\in\operatorname{GF}(2)^n}\hat{1}_{{[1..n]\choose i}}(w)
\sum_{u,v\in \operatorname{GF}(2)^n}1_C(u)(-1)^{u\cdot w}1_C(v)(-1)^{v\cdot w}\\
&=&\frac{1}{2^n|C|}\sum_{w\in\operatorname{GF}(2)^n}\hat{1}_C(w)^2\hat{1}_{{[1..n]\choose i}}(w) \\ 
&=& \frac{1}{2^n|C|}\sum_{w\in\operatorname{GF}(2)^n}\hat{1}_C(w)^2P_i(\wt(w)), 
\end{eqnarray*} 
where the last identity follows from Theorem \ref{thm-july163}. 
It follows from Theorem \ref{thm-july164} that
\begin{eqnarray*}
B'_k &=& \frac{1}{|C|}\sum_{i=0}^n P_k(i)B_i \\ 
&=& \frac{1}{2^n|C|^2}\sum_{w\in\operatorname{GF}(2)^n}\hat{1}_C(w)^2
\sum_{i=0}^n P_k(i)P_i(\wt(w)) \\
&=& \frac{1}{|C|^2}\sum_{w\in\operatorname{GF}(2)^n}\hat{1}_C(w)^2\delta_{k,\wt(w)} \\
&=& \frac{1}{|C|^2}\sum_{w\in{[1..n]\choose k}}\hat{1}_C(w)^2,
\end{eqnarray*}
where $\delta$ is the Kronecker delta function.
\end{proof}

The following follows from Lemma \ref{lem-kor1} and Theorem \ref{thm-Dels1} directly. 

\begin{theorem}\label{thm-joint1}
The set $C \subseteq \gf(2)^n$ (viewed as a matrix) is an orthogonal array of size $|C|$, $n$ constraints, $2$ levels, 
strength $t$ and index $|C|/2^t$ if and only if 
$\hat{1}_{C}(w)=0$ for all $w\in \cup_{j=1}^t \binom{[1..n]}{j}$.
\end{theorem} 

Theorem \ref{thm-joint1} is a spectral characterisation of orthogonal arrays of two levels 
and is a slightly extended version of a special case (the Hamming scheme case) of Delsarte's 
characterisation of $T$-designs in association schemes in general \cite[Theorem 3.10]{Dels73}. 
Since orthogonal arrays and combinatorial $t$-designs are different, Theorem \ref{thm-joint1} 
is clearly different from Theorem \ref{thm-maincharacterisationthm}, which is the main contribution of this paper. Notice that Theorem \ref{thm-joint1} is not meant to be a new 
result.     

\subsection{A spectral extension of Delsarte's characterisation of $t$-designs}\label{sec-Dels2}

Let $C$ be a subset of $\binom{[1..n]}{k}$. The inner distribution $B_i$ of $C$ is defined by 
\[
B_i =\frac{1}{|C|}|\{(u,v ) \in C \times C : |u \cap v| = k-i\}|.
\] 
The outer distribution $B^{'}_l$ of $C$ is defined by  
\[
B'_l=\frac{1}{|C|}\sum_{j=0}^n\frac{\mu_l}{v_j} Q_j(l)B_j \text{ for } l\in[0..n],
\]
where $Q_l(x)$, called the Eberlein polynomial, is
\[
Q_l(x)=\sum_{j=0}^l(-1)^j{x\choose j}{k-x\choose l-j}{n-k-x\choose l-j},
\]
where $v_l={k\choose l}{n-k\choose l}$ and $\mu_l=\frac{n-2l+1}{n-l+1}{n\choose l}$
for $l\in[0..k]$.

The following is Delsarte's characterisation of combinatorial $t$-designs \cite[Theorem 3.10]{Dels73} 
and is a refined version given in \cite[Chapter 21, Theorem 15]{MS77}. 

\begin{theorem}\label{thm-Dels2}
Let $C$ be a subset of $\binom{[1..n]}{k}$. 
The incidence structure $([1..n], C)$ is a $t$-design 
if and only if $B'_1=B'_2=\cdots=B'_t=0$.
\end{theorem} 

This characterisation is not a spectral characterisation. Below we extend it into a spectral 
characterisation. To this end, we need the next lemma. 

\begin{lemma}\label{lem-kor2}
Let $C$ be a nonempty subset of $\binom{[1..n]}{k}$. Then
\begin{eqnarray*}
B_i &=& \frac{1}{2^n|C|}\sum_{w\in\operatorname{GF}(2)^n}
\hat{1}_{C}(w)^2P_{2i}(\wt(w)),\\
B'_i &=& \frac{1}{2^n|C|^2}\sum_{w\in\operatorname{GF}(2)^n}\hat{1}_{C}(w)^2
\sum_{j=0}^n\frac{\mu_i}{v_j} Q_j(i)P_{2j}(\wt(w)).
\end{eqnarray*}
\end{lemma}

\begin{proof}
We have that
\begin{eqnarray*}
B_i &=& \frac{1}{|C|}\sum_{u,v\in \operatorname{GF}(2)^n}1_C(u)1_C(v)1_{{[1..n]\choose 2i}}(u+v)\\
&=&\frac{1}{2^n|C|}\sum_{u,v\in \operatorname{GF}(2)^n}1_C(u)1_C(v)
\sum_{w\in\operatorname{GF}(2)^n}\hat{1}_{{[1..n] \choose 2i}}(w)(-1)^{(u+v)\cdot w}\\
&=&\frac{1}{2^n|C|}\sum_{w\in\operatorname{GF}(2)^n}\hat{1}_{{[1..n]\choose 2i}}(w)
\sum_{u,v\in \operatorname{GF}(2)^n}1_C(u)(-1)^{u\cdot w}1_C(v)(-1)^{v\cdot w}\\
&=&\frac{1}{2^n|C|}\sum_{w\in\operatorname{GF}(2)^n}\hat{1}_{C}(w)^2\hat{1}_{{[1..n]\choose 2i}}(w)\\
&=&\frac{1}{2^n|C|}\sum_{w\in\operatorname{GF}(2)^n}\hat{1}_{C}(w)^2P_{2i}(\wt(w)), 
\end{eqnarray*} 
where the last identity follows from Theorem \ref{thm-july163}. 
It follows that
\begin{align*}
B'_i =\frac{1}{|C|}\sum_{j=0}^n \frac{\mu_i}{v_j}Q_j(i)B_j
=\frac{1}{2^n|C|^2}\sum_{w\in\operatorname{GF}(2)^n}\hat{1}_{C}(w)^2
\sum_{j=0}^n\frac{\mu_i}{v_j} Q_j(i)P_{2j}(\wt(w)).
\end{align*} 
This completes the proof. 
\end{proof} 

The following follows from Lemma \ref{lem-kor2} and Theorem \ref{thm-Dels2} directly. 

\begin{theorem}\label{thm-delstdesign}
The incidence structure $\bD=([1..n], C)$ is a $t$-design 
if and only if 
\begin{eqnarray}\label{eqn-delsartech}
\sum_{w\in\operatorname{GF}(2)^n}\hat{f}_{\mathbb{D}}(w)^2
\sum_{j=0}^n\frac{\mu_i}{v_j} Q_j(i)P_{2j}(\wt(w))=0 
\end{eqnarray}
for all $i\in[1..t]$.
\end{theorem} 

Theorem \ref{thm-delstdesign} is another spectral characterisation of combinatorial 
$t$-designs, and is a slightly extended version of Delsarte's characterisation given in 
Theorem \ref{thm-Dels2}. Since the $t$ equations in \eqref{eqn-delsartech} look very complex, 
the characterisation of Theorem \ref{thm-delstdesign} is complex. In contrast, the spectral characterisation in Theorem \ref{thm-maincharacterisationthm} is 
much simpler. Another difference between the two characterisations is that the spectral characterisation of Theorem \ref{thm-delstdesign} does not involve the parameter $k$ 
and $\lambda$ of a $t$-$(n, k, \lambda)$ directly, while the characterisation of Theorem \ref{thm-maincharacterisationthm} does.

\subsection{Relative $t$-designs and Delsarte's Assmus-Mattson Theorem}\label{sec-Dels3}  

In this subsection, we follow the notation of \cite[Chapter 21]{MS77}. 
Let $(X, R):=(\gf(2)^n,R)$ be the Hamming scheme. Then 
\begin{itemize}
\item $P$ is the first eigenmatrix and $(P)_{ij}=P_j(i)$ ($P_j(x)$ is the Krawtchouk polynomial) \cite[p. 657]{MS77}, 
\item $Q$ is the second eigenmatrix and $Q=P$ \cite[p. 657]{MS77},
\item $J_i$'s are primitive idempotent and $(J_i)_{x,y}=\frac{1}{2^n}P_i(\wt(x+y))$ 
      \cite[p. 657]{MS77},
\item $\rank(J_i)={n\choose i}$ \cite[p. 654]{MS77}. 
\end{itemize} 
In this subsection, we identify a vector $x \in \gf(2)^n$ with its support $\support(x)$.  

Let $T=\{1,2,\ldots,t\}$, where $t \leq n$.
A subset $D$ of $\operatorname{GF}(2)^n$ in the Hamming scheme $(X, R)$ is called a \emph{relative $T$-design with respect to $x\in\operatorname{GF}(2)^n$} 
provided that
\begin{eqnarray}\label{eqn-may181}
|D|(BQ)_i\rank(J_i)= \left((2^nJ_i\chi_D)_x\right)^2 \text{ for all } i\in T,
\end{eqnarray}
where $\chi_D$ is the characteristic vector of $D$, i.e., $\chi_D(x)=1$ when $x\in D$,
and $\chi_D(x)=0$ otherwise,  
$B=(B_0,B_1,\ldots,B_n)$ is the distance (i.e., inner) distribution of $D$, $(BQ)_i$ was defined 
in (\ref{eqn-may191}) and will be given in a more specific form below, and the RHS of 
(\ref{eqn-may181}) will be defined below.   

The forgoing definition is derived from Lemma $2.5.1$-$(iii)$ of \cite{BCN}, 
and \eqref{eqn-may181} is reformulated as follows:
$$(BQ)_i=(BP)_i=\sum_{j=0}^nB_j(P)_{ji}=\sum_{j=0}^nP_i(j)B_j=|D|B'_i,$$ 
where $B'_i$ is the outer distribution of $D$, and is given by 
$$ 
B'_i=\frac{1}{|D|^2}\sum_{w\in {[1..n]\choose i}}\hat{1}_D(w)^2, 
$$  
so that the LHS of \eqref{eqn-may181} is 
$$
|D|(BQ)_i\rank(J_i)={n \choose i}\sum_{w\in {[1..n]\choose i}}\hat{1}_D(w)^2.
$$
The RHS of \eqref{eqn-may181} is 
$$
\left((2^nJ_i\chi_D)_x\right)^2
=\left(2^n\sum_{y\in\operatorname{GF}(2)^n}(J_i)_{xy}\chi_D(y)\right)^2
=\left(\sum_{y\in D}P_i(\wt(x+y))\right)^2.
$$ 
Therefore, \eqref{eqn-may181}  is equivalent to 
\begin{eqnarray}\label{eqn-may182}
{n \choose i}\sum_{w\in {[1..n]\choose i}}\hat{1}_D(w)^2=\left(\sum_{y\in D}P_i(\wt(x+y))\right)^2, \ i \in \{1,2, \ldots, t\}.
\end{eqnarray} 

The following theorem is known in the literature and may be derived from 
\cite[Theorem 6.2]{DelsSeid}. Below we provide a direct proof of it using 
Theorem \ref{thm-maincharacterisationthm}. Recall that we identify a vector 
in $\gf(2)^n$ with its support which is a subset of $[1..n]$ throughout 
this section.  

\begin{theorem}\label{thm-may183} 
Let $D$ be a subset of ${[1..n]\choose k}$. Then $D$ is a t-design in the Johnson scheme  
if and only if $D$ is a relative $T$-design in the Hamming scheme with respect to $0$, where $T=\{1,2,\ldots,t\}$. 
\end{theorem}

\begin{proof}
Let $D$ be a $t$-design in the Johnson scheme.   
By the $``$if part" of Theorem \ref{thm-maincharacterisationthm},  
the LHS of \eqref{eqn-may182} is ${n \choose i}^2\hat{1}_D(w)^2$.
Putting $x=0$, the RHS of \eqref{eqn-may182} is 
$$\left(\sum_{y\in D}P_i(\wt(y))\right)^2=\left(|D|\frac{{n\choose i}}{{n\choose k}}P_k(i)\right)^2.$$ 
Thus $D$ is a relative $T$-design in the Hamming scheme with respect to $0$ if and only if  
$\hat{1}_D(w)^2=(\frac{\lambda}{{n-t\choose k-t}}P_k(i))^2$, where $\wt(w)=i$, $(i=1,2,\ldots,t)$.
By the $``$only part" of Theorem \ref{thm-maincharacterisationthm}, this is indeed true. 
Consequently, $D$ is a relative $T$-design in the Hamming scheme with respect to $0$. 

We now prove the conclusion in the other direction. Let $D$ be a relative $T$-design 
in the Hamming scheme with respect to $0$. 
We have then 
\[
{n \choose i}\sum_{w\in {[1..n]\choose i}}\hat{1}_D(w)^2=\left(|D|P_i(k)\right)^2 
\]
for all $i\in[1..t]$. 
By the Cauchy-Schwartz inequality, we have
\begin{eqnarray*}
{n \choose i}\sum_{w\in {[1..n]\choose i}}\hat{1}_D(w)^2
&=& \sum_{w\in {[1..n]\choose i}}1^2\sum_{w\in {[1..n]\choose i}}\hat{1}_D(w)^2\\
&\geq& \left(\sum_{w\in {[1..n]\choose i}}\hat{1}_D(w)\right)^2=\left(|D|P_i(k)\right)^2.
\end{eqnarray*}
By assumption, the equality holds, and so $\hat{1}_D(w)$ is a constant for all $w$ of weight $i$ for $i\in[1..t]$. 

We denote by $e_j$, $j=1,2,\ldots,n$ the standard basis for $\operatorname{GF}(2)^n$.
For $i=1$, we have that
\[
\hat{1}_D(e_j)=\sum_{y\in D}(-1)^{y_j}=\sum_{y\in D}(1-2y_j)=|D|-2\sum_{y\in D}y_j, 
\] 
which is a constant for all $w$ of weight one. This means that
$D$ is a $1$-design. 

For $i=2$, we have that for $j_1\neq j_2$,
\begin{eqnarray*}
\hat{1}_D(e_{j_1}+e_{j_2}) 
&=& \sum_{y\in D}(-1)^{y_{j_1}+y_{j_2}} \\ 
&=& \sum_{y\in D}(1-2y_{j_1})(1-2y_{j_2})\\
&=& |D|-2\sum_{y\in D}(y_{j_1}+y_{j_2})+4\sum_{y\in D}y_{j_1}y_{j_2}, 
\end{eqnarray*} 
which is a constant for all $w$ of weight two. This means that
$D$ is a $2$-design. In this way, we can prove that $D$ is a $t$-design by induction on $i$. 
\end{proof} 

The following is then a corollary of Theorem \ref{thm-may183}. 

\begin{corollary}\label{cor-May231} 
Let $D$ be a subset of ${[1..n]\choose k}$. Then $([1..n], D)$ is a $t$-$(n, k, \lambda)$ 
design if and only if 
\begin{eqnarray}\label{eqn-may231}
{n \choose i}\sum_{w\in {[1..n]\choose i}}\hat{1}_D(w)^2=\left(\sum_{y\in D}P_i(\wt(y))\right)^2, \ i \in \{1,2, \ldots, t\}.
\end{eqnarray}
\end{corollary}  

Corollary \ref{cor-May231} gives the third spectral characterisation of combinatorial 
$t$-$(n, k, \lambda)$ designs, which is different from Theorem \ref{thm-maincharacterisationthm}. It does not involve the parameters $k$ and $\lambda$ directly. 
Clearly, the spectral characterisation of Theorem \ref{thm-maincharacterisationthm} 
is much simpler.

A relative $t$-designs in a $Q$-polynomial association scheme is a relative $T$-design 
in the scheme, where $T=\{1,2 \ldots, t\}$   
\cite[Chapter 2]{BCN}. 
The following is called Delsarte's Assmus-Mattson Theorem (see \cite[Theorem 8.4]{Dels77}  
and \cite[Theorem $2.8.1$ ]{BCN} for information). 

\begin{theorem}[Delsarte's Assmus-Mattson Theorem]\label{thm-Delsarte'sAM} 
Let $Y$ be a $t$-design in a $Q$-polynomial association scheme $(X, R)$,
and assume that $Y_i:= \{y\in Y : (x,y)\in R_i\}$ is nonempty for $s$ nonzero
values of $i$. Then each $Y_i$ is a relative $(t + 1 - s )$-design with respect to $x$.
\end{theorem} 

Note that Theorem \ref{thm-Delsarte'sAM} gives only a sufficient condition for  
relative $t$-designs, while Theorem \ref{thm-maincharacterisationthm} of this 
paper presents 
a necessary and sufficient condition for combinatorial $t$-designs. Thus, it is   
impossible to derive Theorem \ref{thm-maincharacterisationthm} from Delsarte's 
Assmus-Mattson Theorem (i.e., Theorem \ref{thm-Delsarte'sAM}). In particular,  
it is impossible for Theorem \ref{thm-maincharacterisationthm} to be a special 
case of  Delsarte's Assmus-Mattson Theorem.

\end{document}